\documentclass[a4paper,reqno]{amsart}
\usepackage{amsfonts}
\usepackage{amssymb}
\usepackage{times}
\usepackage{xcolor}
\usepackage{xfrac}
 \usepackage{cancel}

\usepackage{tasks}
\usepackage{xcolor}
\usepackage{tkz-graph}
\usepackage{hyperref}
\usepackage{tikz,pgf}
\usepackage{cite}
\usepackage{fullpage}

\def\a{{\alpha}}

\def\b{{\beta}}

\def\d{{\delta}}

\def\g{{\gamma}}

\def\l{{\lambda}}

\def\gg{{\mathfrak g}}
\def\cal{\mathcal }

\makeatletter
\def\Ddots{\mathinner{\mkern1mu\raise\p@
\vbox{\kern7\p@\hbox{.}}\mkern2mu
\raise4\p@\hbox{.}\mkern2mu\raise7\p@\hbox{.}\mkern1mu}}
\makeatother

\theoremstyle{plain}%

\newtheorem{thm}{Theorem}[section]
\newtheorem{cor}[thm]{Corollary}
\newtheorem{lem}[thm]{Lemma}
\newtheorem{defn}[thm]{Definition}
\newtheorem{exa}[thm]{Example}
\newtheorem{prop}[thm]{Proposition}

\newtheorem{rem}[thm]{Remark}
\newtheorem{con}[thm]{Conjecture}

\newfont{\hueca}{msbm10}

\begin{document}


\def\a{{\alpha}}

\def\b{{\beta}}

\def\d{{\delta}}

\def\g{{\gamma}}

\def\l{{\lambda}}

\def\gg{{\mathfrak g}}
\def\cal{\mathcal}


\title{Cohomological rigidity of solvable Lie algebras of maximal rank}

\author{B.A. Omirov}
\address{Bakhrom A. Omirov \newline \indent
Institute for Advanced Study in Mathematics, Harbin Institute of Technology, Harbin 150001 \newline \indent
Suzhou Research Institute, Harbin Institute of Technology, Harbin 215104, Suzhou, China}
\email{omirovb@mail.ru}

\author{G.O.Solijanova}
\address{Gulkhayo O. Solijanova \newline \indent
National Universityof Uzbekistan named after Mirzo Ulugbek, Uzbekistan}
\email{gulhayo.solijonova@mail.ru}

\author{G.Kh. Urazmatov}
\address{Gulmurod Kh. Urazmatov \newline \indent
V.I. Romanovskiy Institute of Mathematics, Uzbekistan Academy of Sciences, Uzbekistan}
\email{gulmurod0405@gmail.com}

\begin{abstract} We study the second cohomology group with coefficients in the adjoint module for a class of solvable Lie algebras $\mathcal{R}_{\mathcal{T}}$ that arise as maximal solvable extensions of nilpotent Lie algebras $\mathcal{N}$ of maximal rank. Under suitable structural assumptions on the root system determined by the action of a maximal torus $\mathcal{T}$ on $\mathcal{N}$, we obtain sufficient conditions for the cohomological rigidity of $\mathcal{R}_{\mathcal{T}}$. Conversely, we identify explicit configurations of roots that force the second cohomology group to be non-trivial, thereby producing broad families of solvable Lie algebras that are not cohomologically rigid. Our results extend the classical sufficient conditions of Leger and Luks, and they provide a unified and computationally effective framework for determining the cohomological rigidity of a wide class of solvable Lie algebras, including several known results.
\end{abstract}

\maketitle


\section{Introduction}

The study of the second cohomology group of a Lie algebra with coefficients in the adjoint module lies at the intersection of several central themes in Lie theory, including cohomological algebra, deformation theory, the geometry of the variety of Lie algebra laws, and the extensions\cite{Campoamor-Stursberg, Grunewald,Grunewald2}. 
This cohomological invariant plays a rich and unifying role: it governs the space of infinitesimal deformations of a Lie algebra $\mathfrak{g}$, detects geometric rigidity through the vanishing of cohomology, and provides invariant capable of distinguishing non-isomorphic algebras. Its framework introduced the modern viewpoint on the deformation theory of algebraic structures, establishing the cohomological conditions under which a Lie algebra admits nontrivial perturbations of its bracket \cite{Goze, libroKluwer}. 

The deformation theory introduced by Gerstenhaber \cite{gerstenhaber} for associative algebras was later adapted to the context of Lie algebras in the  foundational work of Nijenhuis and Richardson \cite{Nijenhuis}. Within this framework, the cohomology group \(H^{2}(\mathfrak{g}, \mathfrak{g})\) acquires a natural interpretation as the space governing infinitesimal deformations of the Lie algebra \(\mathfrak{g}\). More precisely, a formal one-parameter deformation of the Lie bracket may be expressed as
$$[x,y]_t = [x,y] + t\,\varphi_1(x,y) + t^2\,\varphi_2(x,y) + \cdots,$$
where the first-order term \(\varphi_1\colon \mathfrak g\times\mathfrak g \to \mathfrak g\) is necessarily a $2$-cocycle in the Chevalley--Eilenberg complex. Two deformations are equivalent if and only if their first-order terms differ by a coboundary. Consequently, the second cohomology group \(H^{2}(\mathfrak g,\mathfrak g)\) parametrizes infinitesimal deformations of \(\mathfrak g\). In particular, if this group vanishes, then every formal deformation is equivalent to the original structure; that is, \(\mathfrak g\) is \emph{cohomologically rigid} \cite{Richardson2}.

A second and equally fundamental connection arises from the geometry of the algebraic variety of Lie structures. For a fixed dimension \(n\), all Lie algebra structures on \(\mathbb{F}^n\) form an affine algebraic variety \(\mathcal{L}_n\), on which the group \(\mathrm{GL}_n\) acts naturally by change of basis. Each orbit corresponds to an isomorphism class, while degenerations of Lie algebras correspond to inclusions between Zariski closures of these orbits. The relationship between algebraic deformations and geometric degenerations was clarified in the work of Lauret and Weimar-Woods \cite{Lauret, Weimar}, where it is shown that formal deformations fit naturally within the larger geometric theory of orbit closures.

Recall that every algebraic variety decomposes uniquely into finitely many irreducible components, and a closures of Zariski-open subset forms irreducible component of the variety. Hence, the classification of $\mathcal{L}_n$ is closely tied to identifying those Lie algebras whose orbits are Zariski-open; these are precisely the \emph{rigid} Lie algebras. By Noetherianity, only finitely many such orbits exist. Identifying rigid Lie algebras is therefore, crucial for understanding the global geometry of $\mathcal{L}_n$.

A striking theorem of Nijenhuis and Richardson establishes that cohomological rigidity implies rigidity \cite{Richardson2}, while the existence of rigid Lie algebras that are not cohomologically rigid demonstrates that the converse fails in general (see \cite{Richardson} and \cite{Ancochea2}). Thus, a cohomologically rigid Lie algebra cannot arise as a degeneration of any non-isomorphic algebra. This result places the group $H^{2}(\mathfrak g,\mathfrak g)$ at the center of degeneration theory: the presence of nontrivial cohomology classes signals potential geometric instability, while their absence guarantees the opposite.

Consequently, many rigid and structurally significant solvable Lie algebras, such as maximal solvable extensions of filiform Lie algebras, the model nilpotent Lie algebras, and various Borel-like solvable algebras are characterized precisely by the vanishing of their adjoint second cohomology group; see, for instance, \cite{Ancochea4, Ancochea7, Leger2}.

A particularly rich class of Lie algebras in which these questions converge consists of maximal solvable extensions of nilpotent Lie algebras of maximal rank. Thanks to the result of \cite{super} on the uniqueness of maximal solvable extensions for certain nilpotent Lie algebras, the study of maximal solvable extensions of nilpotent Lie algebras of maximal rank may be reduced to the analysis of solvable Lie algebras of maximal rank. These algebras admit a canonical description as semidirect products
$$\mathcal{R}_{\mathcal{T}} \;=\; \mathcal{N} \rtimes \mathcal{T},$$
where $\mathcal{N}$ is a nilpotent Lie algebra of maximal rank and $\mathcal{T}$ is a maximal torus acting diagonally on $\mathcal{N}$. From the well-known consequence of Mostow’s result \cite{Mostow} on the conjugacy, by an inner automorphism of $\mathcal{N}$, of any two maximal tori $\mathcal{T}_1$ and $\mathcal{T}_2$, it follows that the corresponding solvable extensions are isomorphic, that is, $\mathcal{R}_{\mathcal{T}_1} \cong \mathcal{R}_{\mathcal{T}_2}.$ A practical method for constructing a maximal torus is provided in \cite{superLeibniz}. Consequently, for a given nilpotent (non-characteristically nilpotent) Lie algebra, the potential difficulty of explicitly constructing the solvable algebra $\mathcal{R}_{\mathcal{T}}$ can be disregarded. A key structural feature of such algebras is that $\mathcal{N}$ decomposes into $\mathcal{T}$-weight spaces, providing a toral grading that plays a central role in their cohomological analysis. The cohomological behavior of these algebras depends intricately on the configuration of roots associated with the $\mathcal{T}$-action, and understanding the interactions among these roots has become a fundamental problem. 

An effective and powerful tool for computing the cohomology of solvable Lie algebras is the Hochschild--Serre factorization theorem. In our setting, it reduces the computation of \(H^{n}(\mathcal{R}_{\mathcal{T}},\mathcal{R}_{\mathcal{T}})\) to the analysis of the 
\(\mathcal{T}\)-invariant cohomology groups \(H^{j}(\mathcal{N},\mathcal{R}_{\mathcal{T}})^{\mathcal{T}}\) for all \(j \leq n\). Since solvable Lie algebras \(\mathcal{R}_{\mathcal{T}}\) of maximal rank are complete (i.e., centerless and admitting no outer derivations) \cite{Meng}, we obtain the following equivalence:
$$
\mathcal{R}_{\mathcal{T}} \text{ is cohomologically rigid}
\quad \Longleftrightarrow \quad
H^{2}(\mathcal{N},\mathcal{R}_{\mathcal{T}})^{\mathcal{T}} = 0.
$$
Consequently, the question of rigidity for these algebras is entirely governed by the  $\mathcal{T}$-invariant component of the second cohomology of the nilradical with coefficients in the solvable algebra.

The classical work of Leger and Luks established sufficient conditions for the vanishing of \(H^{2}(\mathfrak{g},\mathfrak{g})\) for certain solvable Lie algebras, formulated in terms of combinatorial relations among the roots arising in the toral decomposition of $\mathcal{N}$ \cite{Leger2}. Their approach, which exploits explicit linear dependencies among weights, has proven both widely applicable and conceptually clear.

The main objective of the present paper is to develop a general framework for computing the second adjoint cohomology group of solvable Lie algebras of the form $\mathcal{R}_{\mathcal{T}}$, and to derive explicit and conceptually natural sufficient condition for its vanishing. Our methods extend the Leger–Luks'  approach, enabling its application to a wider class of solvable Lie algebras and yielding new families of cohomologically rigid structures. 

To complement our cohomological rigidity results, we also establish sufficient conditions ensuring the non-vanishing of \(H^{2}(\mathfrak{g},\mathfrak{g})\). These sufficient conditions clarify the precise situations in which our vanishing results fail to apply. Moreover, an analysis of the patterns emerging from our computations leads us to formulate a conjecture concerning a lower bound for the dimension  of the second adjoint cohomology group of solvable Lie algebras of maximal rank.

The paper is organized as follows. Section~\ref{sec3} is devoted to establishing sufficient conditions by extending the classical Leger-Luks' sufficient condition for the vanishing of the second cohomology group of solvable Lie algebras of maximal rank.  These conditions  apply to a wide class of algebras, including all maximal solvable extensions of nilpotent Lie algebra of maximal rank  appearing in low-dimensional classifications up to dimension nine \cite{Ancochea7}, the model nilpotent and model filiform Lie algebras \cite{Ancochea4,libroKluwer}. 
In addition, we prove that for solvable Lie algebras $\cal R_{\cal T}$  whose nilradical has nilindex \(s+1\), the vanishing of the cohomology groups
$H^{i}(\mathcal R_{\mathcal T},\mathcal R_{\mathcal T})$ for $0\le i\le s-1$ already implies the triviality of the entire adjoint cohomology. As an application, we obtain a new and unified proof of the rigidity of maximal solvable extensions of a subclass of the model nilpotent Lie algebra introduced in \cite{Ancochea6}.

Section~\ref{sec4} focuses on nilradicals of rank two, generated by two primitive weights. We establish a sufficient condition: the absence of the weights $3\alpha_1+\alpha_2$ and $\alpha_1+3\alpha_2$ (see Theorem~\ref{thm4.8}) for cohomological rigidity of $\cal R_{\cal T}$. This condition applies to a broad class of algebras (including, for instance, Example \ref{exa3.7}) that could not be handled by existing methods.

Section~\ref{sec5} turns to the opposite phenomenon: non-rigidity. We establish easily verifiable sufficient conditions on the weight system (see, Theorem~\ref{thm5.1}) that force $H^2(\mathcal{R}_{\mathcal{T}},\mathcal{R}_{\mathcal{T}}) \neq 0$. These conditions are illustrated by explicit examples given in the section, and a lower bound on the dimension of the second cohomology is given. We also provide an example demonstrating that our conditions, while sufficient, are not necessary. This example clarifies the boundary of the current general techniques and highlights the need for further structural insight.

Throughout this work, algebras and modules are assumed to be finite-dimensional and defined over the field $\mathbb{C}.$

\section{Preliminaries}

In this section we provide definitions and preliminary results that will be used throughout of the paper.

Let $\mathcal{G}$ be a Lie algebra and $\mathcal{M}$ a $\mathcal{G}$-module. The Chevalley–Eilenberg cohomology of $\mathcal{G}$ with coefficients in $\mathcal{M}$ is defined by the cochain complex  
$$C^0(\mathcal{G},\mathcal{M})=\mathcal{M},\qquad 
C^k(\mathcal{G},\mathcal{M})=Hom(\wedge^k\mathcal{G},\mathcal{M})\quad(k\ge 1),$$
with coboundary differential $d^k\colon C^k(\mathcal{G},\mathcal{M})\to C^{k+1}(\mathcal{G},\mathcal{M})$ given by  
\begin{multline*} (d^k f)(x_1,\dots,x_{k+1})=
\sum_{i=1}^{k+1}(-1)^{i+1}x_i\cdot f(x_1,\dots,\widehat{x}_i,\dots,x_{k+1})\\
+\sum_{1\le i<j\le k+1} -1)^{i+j}f([x_i,x_j],x_1,\dots,\widehat{x}_i,\dots,\widehat{x}_j,\dots,x_{k+1})
\end{multline*}
for $f\in C^k(\mathcal{G},\mathcal{M})$ and $x_1,\dots,x_{k+1}\in\mathcal{G}$ \cite{Chevellay}.  

We denote the spaces of $k$-coboundaries and $k$-cocycles, respectively, by  
$$B^k(\mathcal{G},\mathcal{M})=\operatorname{im} d^{k-1},\qquad 
Z^k(\mathcal{G},\mathcal{M})=\operatorname{ker} d^k$$
with the convention $B^0(\mathcal{G},\mathcal{M})=0$.

The identity $d^{k}\circ d^{k-1}=0$ leads that  
$B^k(\mathcal{G},\mathcal{M})\subseteq Z^k(\mathcal{G},\mathcal{M}).$ The $k$-th cohomology group of $\cal G$ with coefficient in $\cal M$ is therefore, 
$$H^k(\mathcal{G},\mathcal{M})=Z^k(\mathcal{G},\mathcal{M})/B^k(\mathcal{G},\mathcal{M}).$$

Although the computation of cohomology groups is generally complicated, the Hochschild--Serre factorization theorem provides a substantial simplification for certain classes of solvable Lie algebras (see Theorem~13 in \cite{Hochschild}). Below, we present a version of the Hochschild--Serre factorization theorem adapted to our purposes.

\begin{thm}\label{Serre} Let $\cal G=\cal N\rtimes \cal Q$ be a solvable Lie algebra, where $\cal Q$ is Abelian and let $\cal M$ be a $\cal G$-module. Suppose that ${ad_{x}}_{|\cal N}\ (x\in \cal Q)$ and the representation of $\cal Q$ on $\cal M$ are diagonal. Then 
cohomology groups $H^{n}(\cal G, \cal M)$ satisfy the following isomorphism: 
$$H^n(\cal G,\cal M)\cong \sum\limits_{i+j=n}H^i(\cal Q,\mathbb{F})\otimes H^j(\cal N,\cal M)^{\cal Q}, \quad n\geq0,$$
where \begin{equation}\label{eq1.1}
H^j(\cal N,\cal M)^{\cal Q}=\{f\in H^j(\cal N,\cal M) \ | \ (t\cdot f)=0,\ t\in\cal Q\}
\end{equation}
is the space of $\cal Q$-invariant cocycles of $\cal N$ with values in $\cal M$ and the invariance being defined by
$$(t\cdot f)(z_1,z_2,\dots,z_b)=t\cdot f(z_1,z_2,\dots,z_b)-\sum\limits_{s=1}^{b}f(z_1,\dots,[t,z_s],\dots,z_j).$$
\end{thm}

\begin{rem}\label{rm2.2} Taking into account that 
$H^i(\cal Q,\mathbb{F})=\wedge^i \cal Q$ we conclude that $H^n(\cal G,\cal M)=0$ if and only if $H^j(\cal N,\cal M)^{\cal Q}=0$ for all $0\le j\le n$.
\end{rem}

Since the ground field is $\mathbb{C}$, semisimplicity is equivalent to diagonalizability. Consequently, the semisimplicity of the quotient $(\mathcal N \rtimes \mathcal Q)/\mathcal N$ implies that the induced action of $\mathcal Q$ on $\mathcal N$ is diagonalizable. In particular, $\mathcal Q$ acts on $\mathcal N$ via a family of commuting diagonal operators.

We recall that a nilpotent Lie algebra $\mathcal{N}$ is called of maximal rank if $\operatorname{rank}\cal N=\dim(\mathcal{N}/\mathcal{N}^2).$ 

Let $\mathcal R_{\cal T}: = \mathcal N \rtimes \mathcal T$ be a semidirect product of a nilpotent Lie algebra $\mathcal N$ with a maximal torus $\mathcal T$ acting on $\mathcal N$. Consider the weight-space decomposition of $\mathcal N$ relative to $\mathcal T$,
$$\mathcal N = \bigoplus_{\alpha \in W} \mathcal N_{\alpha}.$$

For $f \in H^{j}(\mathcal N, \mathcal M)^{\mathcal T}$, the condition $\mathcal T \cdot f = 0$ implies that
$f(\mathcal N_{\alpha_{1}}, \dots, \mathcal N_{\alpha_{j}}) \subseteq 
\mathcal N_{\alpha_{1} + \cdots + \alpha_{j}}$ for all $\alpha_{k} \in W.$

Let $\Pi = \{\alpha_{1}, \alpha_{2}, \ldots, \alpha_{n}\}$ denote the set of primitive weights. Then every weight $\mu \in W$ can be uniquely expressed in the form: 
$\mu = \sum\limits_{i=1}^{n} r_{i} \alpha_{i},$ where the coefficients $r_{i}$ are non-negative integers. For such a weight $\mu$, we define its length by
$$\operatorname{length}(\mu) = \sum_{i=1}^{n} r_{i}.$$

Below we adapt the assumptions (i)--(iv) given in \cite{Leger2} to the case of $\mathbb{F}=\mathbb{C}$.

\begin{itemize}
\item[(i)]  $\dim \cal T = \dim(\cal N/\cal N^2) = n$;

\item[(ii)] For $\alpha \in W$, $\dim \cal N_{\alpha} = 1$ and, if $\alpha, \beta, \alpha + \beta$ are all in $W$, $[\cal N_{\alpha}, \cal N_{\beta}] = \cal N_{\alpha + \beta}$. It follows that $[x,e_{\alpha}] = \alpha(x)e_{\alpha}, \ x \in \cal T$; 

\item[(iii)] If $\alpha, \beta, \gamma, \delta, \alpha + \gamma, \beta + \delta$ are all in $W$ with $\alpha, \beta$ primitive and unequal and with $\alpha + \gamma = \beta + \delta$, then there is some $\mu \in W$ such that $\delta = \alpha + \mu$, $\gamma = \beta + \mu$ and at least one of the following is satisfied:
\begin{itemize}
\item[Case 1.] $\alpha + \beta \notin W$;

\item[Case 2.] $\alpha + \beta \in W$ but $\alpha + 2\beta \notin W$ and $\mu = \beta + \nu$ for some $\nu \in W$;

\item[Case 3.] $\alpha + \beta \in W$ but $2\alpha + \beta \notin W$ and $\mu = \alpha + \nu$ for some $\nu \in W$.
\end{itemize}
\end{itemize}

Then one of the main result in \cite{Leger2} related to the assumptions given above states: 

If $\cal R_{\cal T}$ satisfies the assumptions $(i)-(iii)$, then $H^2(\cal R_{\cal T},\cal  R_{\cal T})=0.$ 

Note that the equality $\dim \mathcal{T} = \dim(\mathcal{N} / \mathcal{N}^2)$ for a maximal torus $\mathcal{T}$ of $\mathcal{N}$ characterizes $\mathcal{N}$ as a nilpotent Lie algebra of maximal rank \cite{Meng}. Accordingly, the Lie algebra $\mathcal{R}_{\mathcal{T}}$ is referred to as a solvable Lie algebra of maximal rank.

Since for an arbitrary solvable Lie algebra of the form $\cal N\oplus \cal Q$, where $\cal Q$ is complementary subspace to $\cal N$, we have upper bound estimation $\dim \cal Q \leq \dim(\cal N/\cal N^2)$ (see \cite{Snobl}), then the assumption (i) implies that $\cal R_{\cal T}$ is a maximal extension of $\cal N$. 

By applying the result of \cite{super} on the uniqueness, up to isomorphism, of maximal solvable extensions of a certain class of nilpotent Lie algebras, and observing that condition~(i) ensures $\mathcal{N}$ belongs to this class, we conclude that $\mathcal{R}_{\mathcal{T}}$ is the unique (up to isomorphism) maximal solvable extension of $\mathcal{N}$. Consequently, without loss of generality, we may consider $\mathcal{R}_{\mathcal{T}}$ in place of any maximal solvable extension of $\mathcal{N}$.

Further, we shall assume that an algebra $\cal R_{\cal T}$ satisfies conditions (i)–(ii). In what follows, if a weight $\lambda$ admits a decomposition $\lambda=\alpha + \gamma=\beta + \delta$ that satisfies property (iii), we will say that "$\lambda$ satisfies the Leger-Luks' condition".

Below, we present the principal concept of this work.

\begin{defn} \cite{Richardson2}. 
 A Lie algebra $\cal G$ is called cohomologically rigid if $H^2(\cal G,\cal G)=0$.
\end{defn}

By examining the structure constants of a two-step nilpotent Lie algebra of maximal rank, we obtain the following result.

\begin{prop}
Let $\mathcal N$ be a nilpotent Lie algebra of maximal rank satisfying $\mathcal N^3 = 0$. Then $\mathcal R_{\mathcal T}$ is cohomologically rigid.
\end{prop}
\begin{proof} A straightforward analysis of the structure a $2$-cocycle leads that it is a $2$-coboundary.
\end{proof}

From now on, we focus on solvable Lie algebras $\cal R_{\cal T}$ with a nilradical of nilindex greater than 3.

For each weight space $\mathcal N_{\alpha} = \operatorname{Span}\{e_{\alpha}\}$ with $\alpha \in W$, 
we denote by $c(\alpha,\beta)$ the corresponding structure constant defined by the product $[e_{\alpha}, e_{\beta}] = c(\alpha,\beta)\, e_{\alpha+\beta}.$ 

Let $\mathcal{M}$ be a one-dimensional $\mathcal{R}_{\mathcal{T}}$-module, say $\mathcal{M} = \mathbb{C} m$, and let the weight of $\mathcal{T}$ on $\mathcal{M}$ be $\Lambda$.  
Then, for any $f \in H^2(\mathcal{N}, \mathcal{M})^{\mathcal{T}}$, we may write
$$f(e_{\alpha}, e_{\beta}) =
\begin{cases}
\Phi(\alpha,\beta)\, m, & \text{if } \alpha + \beta = \Lambda, \\[1ex]
0, & \text{otherwise}.
\end{cases}
$$

The one-dimensionality of $\mathcal{M}$ implies that the $\mathcal{T}$-invariant second-cocycle condition simplifies, reducing the usual sum of six terms to a sum of three.  
For a triple of elements $\{e_{\alpha}, e_{\beta}, e_{\gamma}\}$, the Jacobi identity and the second-cocycle condition, expressed in terms of the structure constants and the scalars defining $f$, then take the form
\begin{equation}\label{eq2}
\left\{\begin{array}{llll}
c(\alpha,\beta)c(\alpha+\beta,\gamma) + c(\beta,\gamma)c(\beta+\gamma,\alpha) + c(\gamma,\alpha)c(\gamma+\alpha,\beta) &=& 0, \\[3mm]
c(\alpha,\beta)\Phi(\alpha+\beta,\gamma) + c(\beta,\gamma)\Phi(\beta+\gamma,\alpha) + c(\gamma,\alpha)\Phi(\gamma+\alpha,\beta) &=& 0.
\end{array}\right.
\end{equation}

For convenience, we denote by $L(\alpha,\beta,\gamma)$ and $Z(\alpha,\beta,\gamma)$ the left-hand sides of the first and second equalities in \eqref{eq2}, respectively.

\section{Solvable Lie algebras of maximal rank with vanishing second cohomology group.}\label{sec3}

In this section, building upon the Leger–Luks' conditions, we establish sufficient conditions to ensure cohomological rigidity for a broader class of solvable Lie algebras of maximal rank.
Taking into account that solvable Lie algebra $\cal R_{\cal T}$ of maximal rank is complete algebra, an application of Theorem 2 from \cite{Meng} yields the following auxiliary result.

\begin{lem}\label{2.14} For any non-primitive weight $\alpha\in W$ there exists a primitive  weight $\alpha_{i_0}$  such that $\alpha=\beta+\alpha_{i_0}$ for some $\beta\in W$, that is, $c({\alpha_{i_0}, \beta})\neq 0.$
\end{lem}

Let us present the main result of this section.

\begin{thm}\label{thm3.2}Let $\cal R_{\cal T}$ be a solvable Lie algebra with nilradical $\cal N$ and let $\cal M$ be an $\cal R_{\cal T}$-module such that the action of $\cal T$ on $\cal M$ is diagonal and the weight of $\cal T$ on $\cal M$ are in $W$. If for each weight $\lambda$ of $\cal T$ on  $\cal N$ at least one of the following conditions holds true
\begin{itemize}
    \item[$i)$] Leger-Luks' conditions;
    \item[$ii)$] $\lambda=\alpha_{i}+\mu$ for some $\alpha_{i}\in \Pi$ and $\lambda-\alpha_j\notin W$ for all $\alpha_j\in \Pi \setminus \{\alpha_{i}\}$;
    \item[$iii)$] $\lambda=\alpha_{i}+\alpha_{j}$ or $\lambda=\alpha_{i}+(\alpha_{i}+\alpha_{j})$ for some distinct $\alpha_i,\alpha_j\in\Pi$.
\end{itemize}
Then $H^2(\cal N,\cal M)^{\cal T}=0.$
\end{thm}
\begin{proof} 

From the identity
\[
t\cdot (e_{\alpha}\cdot m_{\beta})
= [t,e_{\alpha}]\cdot m_{\beta} + e_{\alpha}\cdot (t\cdot m_{\beta})
= (\alpha+\beta)(t)\, e_{\alpha}\cdot m_{\beta}, 
\qquad e_{\alpha}\in \mathcal N_{\alpha},\; m_{\beta}\in \mathcal M,
\]
it follows that
$\mathcal N_{\alpha}\cdot \mathcal M_{\beta} \subseteq \mathcal M_{\alpha+\beta}.$ Thus, $\mathcal N$ acts on $\mathcal M$ by nilpotent operators. By Engel’s theorem, since $\mathcal M\neq 0$, there exists a nonzero vector $m\in \mathcal M$ such that 
$\mathcal N\cdot m = 0$. Consequently, $\mathcal M$ possesses the proper 
$\mathcal N\rtimes \mathcal T$-submodules: $\mathcal M'=\mathbb C m$ and $\mathcal M/\mathcal M'.$ 

This yields the short exact sequence 
$$H^{2}(\mathcal N,\mathcal M')^{\mathcal T}
\;\longrightarrow\;
H^{2}(\mathcal N,\mathcal M)^{\mathcal T}
\;\longrightarrow\;
H^{2}(\mathcal N,\mathcal M/\mathcal M')^{\mathcal T}.$$

Applying the same argument to the module $\mathcal M/\mathcal M'$ produces another short exact sequence of the same form. Iterating this procedure, we eventually obtain an exact sequence with of one-dimensional modules involved in outer terms. Now, assuming that $H^{2}(\mathcal N, \mathcal V)^{\mathcal T}=0$ for every one-dimensional module $\mathcal V$, the above induction shows that
$H^{2}(\mathcal N, \mathcal M)^{\mathcal T}=0.$

Therefore, in order to prove the theorem, it suffices to consider the case 
where $\mathcal M$ is one-dimensional as a module over $\mathcal N\rtimes \mathcal T$. Let suppose $\cal M=\mathbb Cm$ and $\Lambda$ is the  weight of $\cal T$ on $\cal M$. Further we shall focus on weight $\lambda=\Lambda.$

Let take arbitrary $f\in H^2(\cal N,\cal M)^{\cal T}$ such that $\cal T\cdot f=0$. 
 
\textbf{Case (i)}.  Then the result follows directly from the work \cite{Leger2}
(see, proof of Proposition 5.3, in \cite{Leger2}), where it is constructed $g\in C^1(\cal N,\cal M)$ such that $f=\delta g$.

\textbf{Case (ii)}. Without loss of generality, one can assume that $\Lambda = \alpha_{1} + \mu$ for some weight $\mu$. 

We prove the equality
\begin{equation}\label{eq4}
\Phi(\nu, \Lambda - \nu) = a(\nu)\,\Phi(\alpha_{1}, \Lambda - \alpha_{1}), \quad a(\nu)\in \mathbb{C}
\end{equation}
for every weight $\nu \in W$, by induction on the length of $\nu$. 

For $\mathrm{length}(\nu)=1$, Equality \eqref{eq4} immediately follows from condition (ii). 
Let $\mathrm{length}(\nu)=2$, that is,  $\nu=\alpha_{s}+\alpha_{t}$. 

If $\Lambda-\nu\notin W$ we set $a(\nu)=0$, and \eqref{eq4} holds trivially. Thus, we may assume that $\Lambda-\nu\in W$. In this case we consider 
\begin{equation}\label{eq3}
L(\alpha_{s},\alpha_{t},\Lambda-(\alpha_{s}+\alpha_{t}))
=
Z(\alpha_{s},\alpha_{t},\Lambda-(\alpha_{s}+\alpha_{t}))
=0.
\end{equation}

If $2\le s\neq t \le n$, then \eqref{eq3} implies 
$$c(\alpha_s+\alpha_t,\Lambda-(\alpha_s+\alpha_t))= \Phi(\alpha_s+\alpha_t,\Lambda-(\alpha_s+\alpha_t))=0.$$ 

If $s=1$, then \eqref{eq3} deduces
$$ \begin{array}{ccc}
c(\alpha_1+\alpha_t,\Lambda-(\alpha_1+\alpha_t))= \left\{\begin{array}{ccc}
\frac{c(\alpha_t,\Lambda-(\alpha_1+\alpha_t))}{c(\alpha_1,\alpha_t)}c(\alpha_1,\Lambda-\alpha_1), 
&
\text{if} \quad 
c(\alpha_1,\alpha_t) \neq 0,\\[1mm]
0, & \text{otherwise,}\end{array}\right.
\end{array}$$ 
$$ \begin{array}{ccc}
\Phi(\alpha_1+\alpha_t,\Lambda-(\alpha_1+\alpha_t))= \left\{\begin{array}{ccc}
\frac{c(\alpha_t,\Lambda-(\alpha_1+\alpha_t))}{c(\alpha_1,\alpha_t)}\Phi(\alpha_1,\Lambda-\alpha_1), 
&
\text{if} \quad 
c(\alpha_1,\alpha_t) \neq 0,\\[1mm]
0, & \text{otherwise.}\end{array}\right.
\end{array}$$
Thus, for a suitable scalar $a(\nu)$, we obtain
$\Phi(\nu, \Lambda-\nu) = a(\nu)\,\Phi(\alpha_{1}, \Lambda-\alpha_{1})$ for every weight $\nu$ of length $2$.

Assume that \eqref{eq4} holds for every weight $\nu$ satisfying $\mathrm{length}(\nu)\leq p$, and let $\nu$ be an arbitrary weight of length $p+1$.

If $\Lambda - \nu \notin W$, we set $a(\nu)=0$. Now suppose that $\nu \in W$ and $\Lambda - \nu \in W$. By Lemma~\ref{2.14}, there exist a primitive weight $\alpha$ and a weight $\tau \in W$ such that 
 $$\nu = \alpha + \tau, \qquad [e_{\alpha}, e_{\tau}] \neq 0.
$$

From the equality $Z(\alpha,\tau,\Lambda-\nu)=0$, together with the induction hypothesis, we deduce that the parameters $\Phi(\Lambda-\alpha,\alpha)$ and $\Phi(\Lambda-\tau,\tau)$ are linearly determined by $\Phi(\alpha_{1},\Lambda-\alpha_{1})$. Hence, \eqref{eq3} holds for every $\nu\in W$.

Similarly, we obtain
\begin{equation}\label{eq5}
c(\nu, \Lambda-\nu)=a(\nu)\, c(\alpha_{1},\Lambda-\alpha_{1}).
\end{equation}

From condition~(ii) together with Lemma~\ref{2.14}, we conclude that $
c(\alpha_{1},\Lambda-\alpha_{1}) \neq 0.$ 

Using identities \eqref{eq3} and \eqref{eq5}, we derive
\begin{equation}\label{eq6}
c(\alpha_{1},\Lambda-\alpha_{1})\,\Phi(\nu, \Lambda-\nu)
   = c(\nu, \Lambda-\nu)\,\Phi(\alpha_{1},\Lambda-\alpha_{1}), 
   \qquad \nu \in W.
\end{equation}

Define $g \in C^1(\mathcal{N}, \mathcal{M})$ by
\begin{equation}\label{eq7}g(\rho)=\left\{\begin{array}{cccc}

-\frac{\Phi(\alpha_{1},\Lambda-\alpha_{1})}{c(\alpha_{1},\Lambda-\alpha_{1})}m, & \text{if }\rho=\Lambda, \\[2mm]
0, & \text{if }\rho\neq\Lambda.
\end{array}\right.
\end{equation}

If $\alpha+\gamma\ne \Lambda$, then it is immediate that 
$(f-d^1 g)(e_\alpha, e_{\gamma})=0.$

If $\alpha+\gamma=\Lambda$, then applying \eqref{eq6} with 
$\nu=\alpha$ and $\Lambda -\nu=\gamma$ yields
\[
(f-d^1 g)(e_\alpha, e_{\gamma})
   = f(e_\alpha,e_{\gamma}) + c(\alpha,\gamma)\,g(e_{\Lambda}) = 0.
\]

\textbf{Case (iii)}. Let $\Lambda = \alpha_i + \alpha_j$ or $\Lambda = 2\alpha_i + \alpha_j$, with some distinct  $\alpha_i, \alpha_j \in \Pi,$.  
Then, by defining $g$ as in \eqref{eq7} with $\alpha_1$ replaced by $\alpha_i$, we conclude that $f - d^1 g = 0.$
\end{proof}

Since $H^i(\mathcal{R}_{\mathcal{T}}, \mathcal{R}_{\mathcal{T}}) = 0,\ 0 \le i \le 1$ for any solvable Lie algebra $\cal R_{\cal T}$ (see \cite{super}), it follows from Remark~\ref{rm2.2} that we obtain the following corollary.

\begin{cor}
Let $\mathcal{R}_{\mathcal{T}}$ be a solvable Lie algebra satisfying the conditions of Theorem~\ref{thm3.2}. 
Then $\mathcal{R}_{\mathcal{T}}$ is cohomologically rigid.
\end{cor}

Recall that the cohomological rigidity of solvable Lie algebras of maximal rank in dimensions not exceeding $9$ was established in~\cite{Ancochea7}. In particular, the algebras
$$\mu_5^1, \ \mu_6^2, \ \mu_7^5, \ \mu_7^6, \ \mu_8^{22}, \ \mu_8^{25}, \ \mu_8^{26}, \ \mu_8^{33}$$
from that classification satisfy conditions (ii) and (iii) of Theorem~\ref{thm3.2}. Moreover, the lengthy computations required to establish the cohomological rigidity of the maximal solvable extension of the model nilpotent Lie algebra, as carried out in~\cite{Ancochea4}, can be considerably simplified by applying Theorem~\ref{thm3.2}, since this solvable algebra satisfies both conditions (ii) and (iii).   Similarly, the maximal solvable extensions of the model filiform Lie algebra $n_1$ also satisfy the assumptions of cases (ii) and (iii), and their cohomological rigidity is established in~\cite{libroKluwer}.  

Establishing the cohomological rigidity of solvable Lie algebras typically requires extensive computations. However, by applying Theorem~\ref{thm3.2}, the cohomological rigidity of the above mentioned algebras can be determined in a significantly more efficient manner whenever at least one of the conditions (i)--(iii) is satisfied.

We now proceed with the following results, which will be useful in the study of high-order cohomology groups.

\begin{lem}\label{lem3.4} Let $\mathcal R_{\mathcal T}$ be a solvable Lie algebra with nilradical $\cal N$ of nilindex $(s+1)$. Let $\mathcal M$ be an $\mathcal R_{\mathcal T}$-module on which the action of $\mathcal T$ on $\cal M$ is diagonal, and assume that all weights of $\mathcal T$ on $\mathcal M$ lie in $W$. Then $H^{s}(\mathcal N,\mathcal M)^{\mathcal T} = 0.$
\end{lem}
\begin{proof} In a similar arguments as in the proof of Theorem~\ref{thm3.2}, it suffices to treat the case in which $\mathcal M=\mathbb C m$ is one–dimensional and $\Lambda$ is the weight of $\mathcal T$ on $\mathcal M$.  
Let $f\in H^{s}(\mathcal N,\mathcal M)^{\mathcal T}$. It is clear that it suffices to consider the case where the maximal rank is at least $s$. Then for any $s$-tuple $(\alpha_{i_1},\dots,\alpha_{i_s})$ we have
$$f(e_{\alpha_{i_1}},\dots,e_{\alpha_{i_s}})\in 
\mathcal N_{\alpha_{i_1}+\cdots+\alpha_{i_s}},$$
and since $\mathcal N^{\,s+1}=0$, the only non-zero contributions arise when
all $\alpha_{i_j}\in\Pi.$

Accordingly, it suffices to consider
$$f(e_{\alpha_{i_1}},\dots,e_{\alpha_{i_s}})
=\Phi(\alpha_{i_1},\dots,\alpha_{i_s})\, m,
\qquad \alpha_{i_j}\in\Pi,  \quad \sum\limits_{j=1}^s\alpha_{i_j}=\Lambda
$$
corresponding to the $\mathcal T$-weight $\Lambda$; all other components of $f$ vanish. Since $\Lambda\in W$, for all $s$-tuple $(\alpha_{i_1},\dots,\alpha_{i_s})$ there exist $1\le{p}<{q}\le s$ such that $c(\alpha_{i_p},\alpha_{i_q})\neq0$. Let define $g\in C^{s-1}(\mathcal N,\mathcal M)$ as

$$\left\{\begin{array}{clll}
g(e_{\alpha_{i_p}+\alpha_{i_q}},e_{\alpha_{i_{1}}},\dots,\widehat{e}_{\alpha_{i_p}}, \cdots, \widehat{e}_{\alpha_{i_q}}, \cdots ,e_{\alpha_{i_s}})&=&(-1)^{p+q}\frac{\Phi(e_{\alpha_{i_1}},\dots,e_{\alpha_{i_s}})}{c(\alpha_{i_p},\alpha_{i_q})}m,\\[3mm]
g(\alpha_{j_1},\dots,\alpha_{j_s})&=&0,\quad\mbox{otherwise}.\\[3mm]\end{array}\right.
$$
Then by the choice of $g$ we obtain 
$$
(f-d^{s-1}g)(e_{\alpha_{i_1}},\dots,e_{\alpha_{i_{s}}})=
0.$$
Hence $f=d^{s-1}g$, and the cocycle $f$ is a $s$-coboundary. This proves $H^{s}(\mathcal N,\mathcal M)^{\mathcal T}=0$.
\end{proof}

\begin{prop}\label{prop4.11} Suppose $\cal R_{\cal T}$ satisfies the conditions of Lemma \ref{lem3.4} and  $H^i(\cal R_{\cal T}, \cal R_{\cal T})=0$ for all $0 \leq i \leq s-1$. Then $H(\cal N, \cal M)^{\cal T}=0$.
\end{prop}
\begin{proof} Since $H^i(\cal N, \cal R_{\cal T})^{\cal T}=0$ for all $i\geq s+1$, it is sufficient to prove $H^s(\cal N, \cal R_{\cal T})^{\cal T}=0$, which immediately follows from Lemma \ref{lem3.4}.
\end{proof}

To facilitate the analysis of higher-order cohomology groups, we now present the following result.

\begin{cor}\label{cor3.6}
Suppose $\cal R_{\cal T}$ be a solvable Lie algebra with nilradical $\cal N$ of nilindex $(s+1)$ such that $H^i(\cal R_{\cal T}, \cal R_{\cal T})=0$ for all $0 \leq i \leq s-1$. Then $H(\cal R_{\cal T}, \cal R_{\cal T})=0$.
\end{cor}

Below we present an example that satisfies the assumptions of Corollary \ref{cor3.6}.
\begin{exa}
Let $\mathfrak n$ be the nilpotent Lie algebra with multiplication table
$$\left\{\begin{array}{lllll}
[e_1,e_j]&=&e_{j+1},& 2\le j\le n_1,\\[2mm]
[e_1,e_{n_1+\dots+n_{i}+j}]&=&e_{n_1+\dots+n_{i}+j+1},& 2\le j\le n_{i+1},& 1\le i\le k-1.
\end{array}\right.$$
where $2\le n_i\le 4$ for all $1\le i\le k$.

This Lie algebra, for arbitrary values of the parameters $n_i$, was introduced in \cite{Ancochea6} as the model nilpotent Lie algebra. In that paper it was shown that, for solvable extension $\mathcal R_{\mathcal T}$ of $\mathfrak n$, one has $H^{p}(\mathcal R_{\mathcal T},\mathcal R_{\mathcal T})=0$, $0\leq p \leq 3.$
Then the assertion of Corollary \ref{cor3.6} leads  
$H(\mathcal R_{\mathcal T},\mathcal R_{\mathcal T})=0.$
\end{exa}

\section{Cohomological rigidity of solvable Lie algebras of maximal rank 2}\label{sec4}

In this section, we continue our study of cohomological rigidity for maximal solvable extensions of nilpotent Lie algebras of maximal rank. 
Our attention is restricted to solvable Lie algebras $\mathcal R_{\mathcal T}$ whose nilradical is generated by two elements and has rank two. 
Thus, $\mathcal N = \langle e_{\alpha_{1}},\, e_{\alpha_{2}} \rangle.$ As in the proof of Theorem~\ref{thm3.2}, the general case of an $\mathcal R_{\mathcal T}$-module $\mathcal M$ reduces to the one-dimensional situation. 
Accordingly, we may assume without loss of generality that
$\mathcal M = \mathbb{C} m,$
where $m$ has weight $\Lambda$ with respect to the action of $\mathcal T$.

We denote by $c_{i,j}^{\,s,t}$ the structure constant determined by the product
{\large
\begin{align*}
[e_{i\alpha_{1}+j\alpha_{2}},\, e_{s\alpha_{1}+t\alpha_{2}}]
&= c_{i,j}^{\,s,t}\, e_{(i+s)\alpha_{1} + (j+t)\alpha_{2}}.
\end{align*}
}
and by $\Phi_{i,j}^{\,s,t}$ the coefficient associated with a cocycle 
$f \in H^{2}(\mathcal N,\mathcal M)^{\mathcal T}$, defined by
$$f(e_{i\alpha_{1}+j\alpha_{2}},\, e_{s\alpha_{1}+t\alpha_{2}})
=
\left \{ \begin{array}{ccl}
\Phi_{i,j}^{\,s,t}\, m, 
&& \text{if } (i+s)\alpha_{1} + (j+t)\alpha_{2} = \Lambda, \\[3mm]
0, 
&& \text{otherwise}.
\end{array}\right.$$

In addition, for a weight $\mu = p\alpha_{1} + q\alpha_{2} \in W$, we introduce the notation
$$\mathrm{diff}(\mu) = |p - q|.$$

\begin{lem}\label{lem4.1} Let $\cal R_{\cal T}$ be a solvable Lie algebra with nilradical $\mathcal{N}$ of nilindex at most $5$. Then $H^2(\cal N, \cal M)^{\cal T}=0$.
\end{lem}

\begin{proof} Under the hypothesis on the nilindex, the set of all weights satisfies
$$W \subseteq \{ \alpha_1,\ \alpha_2,\ \alpha_1 + \alpha_2,\ 2\alpha_1 + \alpha_2,\ \alpha_1 + 2\alpha_2,\ 
3\alpha_1 + \alpha_2,\ \alpha_1 + 3\alpha_2,\ 2\alpha_1 + 2\alpha_2 \}.$$

Fix $\Lambda = (i+s)\alpha_{1} + (j+t)\alpha_{2} \in W \setminus \{2\alpha_1 + 2\alpha_2\}$, 
where $i\alpha_1 + j\alpha_2,\ s\alpha_1 + t\alpha_2 \in W$.  
Since the element $e_{\Lambda}$ can be obtained uniquely as the product 
{\large
\begin{align*}
e_{\Lambda} = [e_{i\alpha_1+j\alpha_2}, \, e_{s\alpha_1+t\alpha_2}],
\end{align*}
}
it follows that the associated structure constant $c_{i,j}^{\,s,t}$ is nonzero.

Define a cochain $g \in C^{1}(\mathcal N, \mathcal M)$ by setting
$$g(e_{\Lambda})=-\Phi_{i,j}^{s,t}m/c_{i,j}^{s,t} \quad \mbox{and} \quad g(e_{\rho})=0, \quad \rho\neq \Lambda.$$ 
A straightforward computation then yields $f = d^{1} g.$

Now suppose that $\Lambda = 2\alpha_1 + 2\alpha_2 \in W$.  
The equalities
$$L(\alpha_{1}, \alpha_{2}, \alpha_{1} + \alpha_{2}) = 0,
\qquad
Z(\alpha_{1}, \alpha_{2}, \alpha_{1} + \alpha_{2}) = 0,$$
imply 
$$c_{1,2}^{\,1,0} \, \Phi_{0,1}^{\,2,1}
=
c_{2,1}^{\,0,1} \, \Phi_{1,0}^{\,1,2},
\qquad \text{with} \qquad
c_{1,0}^{\,1,2} c_{0,1}^{\,2,1} \neq 0.
$$

Define again a cochain $g \in C^{1}(\mathcal N, \mathcal M)$ by
$$g(e_{\Lambda})=-{\Phi_{1,0}^{1,2}}m/{c_{1,0}^{1,2}} \quad \mbox{and} \quad g(e_{\rho})=0, \quad \rho\neq \Lambda,$$
It follows once more that $f = d^{1} g.$
In conclusion, every $2$-cocycle is a $2$-coboundary.
\end{proof}

Applying Remark \ref{rm2.2}, we obtain the following corollary.
\begin{cor} Let $\cal R_{\cal T}$ be a solvable Lie algebra with nilradical $\mathcal{N}$ of nilindex at most $5$. Then $\cal R_{\cal T}$ is cohomologically rigid.
\end{cor}

We now present the key lemmas that will be used in the proof of the main theorems. We begin by defining 
for each integer $i\ge 1$, the following sets:
$$W_i=\{\mu\in W \mid \mathrm{length}(\mu)\le i\}, \qquad
V_i=\{\mu\in W \mid \mathrm{length}(\mu)=i\}.$$

\begin{lem}\label{lem4.3} 
Let $3\alpha_1 + \alpha_2$ and $\alpha_1 + 3\alpha_2$ are not in $W$. Then the following properties hold:
\begin{itemize}
    \item[(i)] For any $\mu \in W_6$,  $\mathrm{diff}(\mu) \leq 1$.
    \item[(ii)] For any weights $\tau_1, \tau_2 \in W$ such that $\mathrm{length}(\tau_1) + \mathrm{length}(\tau_2) \leq 6$, if $\tau_1 + \tau_2 \in W$, then the Lie bracket $[e_{\tau_1}, e_{\tau_2}]$ vanishes if and only if $\mathrm{diff}(\tau_1) + \mathrm{diff}(\tau_2) = 0$.
\end{itemize}
\end{lem}
\begin{proof} It is straightforward to verify that the lemma holds for all weights  $\mu \in V_i,\ i\le 3$. By a routine inductive argument together with the observation that $3\alpha_{1} + \alpha_{2}, \alpha_{1} + 3\alpha_{2}$ do not lie in $W$, one deduces that
$$n\alpha_{1} + \alpha_{2}, \quad \alpha_{1} + n\alpha_{2} \notin V_{n+1} \quad \text{for all } n \ge 4.$$

Thus, we need to prove the claim for $V_k$ with $4 \leq k \leq 6$.

{Case $V_4$}: It is clear that  $\mathrm{diff}(\mu) = 0$ for the weight $\mu = 2\alpha_1 + 2\alpha_2$; hence  $\mu$ satisfies property (i). The equality $L( \alpha_1, \alpha_2, \alpha_1 + \alpha_2)=0$ implies that both coefficients $c_{2,1}^{0,1}$ and $c_{1,2}^{1,0}$ are non-zero. Consequently,  $\mathrm{diff}(\tau_1) + \mathrm{diff}(\tau_2) = 2$ in the cases where $(\tau_1,\tau_2)$ is either $(2\alpha_1+\alpha_2,\alpha_2)$ or $(\alpha_1+2\alpha_2,\alpha_1)$. 

{Case $V_5$}: We have $V_5 \subseteq \{3\alpha_1 + 2\alpha_2, 2\alpha_1 + 3\alpha_2\}$. Hence, $\mathrm{diff}(\mu) = 1$ for every $\mu \in V_5$. 
Assume that  $3\alpha_1 + 2\alpha_2 \in V_5$. Since the element $e_{3\alpha_1 + 2\alpha_2}$ can only arise from one of the  brackets $$[e_{\alpha_1}, e_{2\alpha_1 + 2\alpha_2}]\quad [e_{\alpha_1 + \alpha_2}, e_{2\alpha_1 + \alpha_2}],$$ then in both cases we have $\mathrm{diff}(\tau_1) + \mathrm{diff}(\tau_2) = 1$. 

The case $2\alpha_{1} + 3\alpha_{2} \in V_{5}$ verifies the claims of the lemma by symmetry of $\alpha_{1}$ and $\alpha_{2}$.

{Case $V_6$}: Since $V_6 \subseteq \{4\alpha_1 + 2\alpha_2, 2\alpha_1 + 4\alpha_2, 3\alpha_1 + 3\alpha_2\}$, we consider the possible cases.

The equality $L(\alpha_1,\alpha_1 + \alpha_2,2\alpha_1 + \alpha_2)=0$ forces $c^{3,2}_{1,0} = 0$. However, $4\alpha_1 + 2\alpha_2\in V_6$ if and only if $c_{1,0}^{3,2}\neq0$. This contradiction shows $4\alpha_1 + 2\alpha_2 \notin V_6$. By symmetry, we also obtain $2\alpha_1 + 4\alpha_2 \notin V_6$.

Let $3\alpha_1 + 3\alpha_2 \in W_6$. This implies that at least one of the following structure constants is non-zero:
$$c_{1,1}^{2,2},\quad c_{1,0}^{2,3},\quad c_{0,1}^{3,2},\quad c_{2,1}^{1,2}.$$

We consider the Jacobi identity corresponding to various decompositions of the weights and the resulting constraints.
\begin{center}\begin{tabular}{clllllllll}
       \qquad\quad  Jacobi identity& & & &\qquad Constrains\\
        \hline \hline
\\                  
        $L(\alpha_1,\alpha_2,2\alpha_1+2\alpha_2)$&$=$&$0$ &$\Rightarrow $ &    $c_{1,1}^{2,2}c_{1,0}^{0,1}=c_{1,0}^{2,3}c_{0,1}^{2,2}+c_{0,1}^{3,2}c_{2,2}^{1,0},$\\
        \\$L(\alpha_1,\alpha_1+\alpha_2,\alpha_1+2\alpha_2)$&$=$&$0$ & $\Rightarrow $ &             $c_{1,1}^{2,2}c_{1,0}^{1,2}=c_{1,0}^{2,3}c_{1,1}^{1,2}+c_{1,2}^{2,1}c_{1,0}^{1,1},$\\
        
        \\$ L(\alpha_2,\alpha_1+\alpha_2,2\alpha_1+\alpha_2)$&$=$&$0$&$\Rightarrow $&          $c_{1,1}^{2,2}c_{0,1}^{2,1}=c_{0,1}^{3,2}c_{1,1}^{2,1}+c_{2,1}^{1,2}c_{0,1}^{1,1},$\\
        
        \\$L(\alpha_1,\alpha_2,\alpha_1+2\alpha_2)$&$=$&$0$ &$\Rightarrow $ &         $c_{1,1}^{1,2}c_{1.0}^{0,1}+{c_{1,2}^{1,0}c_{2,2}^{0,1}}=0,$\\
        
       \\$L(\alpha_1,\alpha_2,2\alpha_1+\alpha_2)$&$=$&$0$ & $\Rightarrow $ &   $c_{1,1}^{2,1}c_{1.0}^{0,1}+c_{0,1}^{2,1}c_{2,2}^{1,0}=0,$\\
        \\$L(\alpha_1,\alpha_2,\alpha_1+\alpha_2)$&$=$&$0$ & $\Rightarrow $ &        $c_{1,1}^{0,1}c_{1,0}^{1,2}+c_{1,1}^{1,0}c_{2,1}^{0,1}=0.$\\
\end{tabular}
\end{center}

The analysis of obtained constraints shows that $c_{1,1}^{2,2} = 0$. Moreover, if any one of the constants $c_{1,0}^{2,3}$, $c_{0,1}^{3,2}$, or $c_{2,1}^{1,2}$ vanishes, then all must vanish. This would contradict the assumption that $3\alpha_{1} + 3\alpha_{2} \in V_{6}$. Hence, all three constants are nonzero, which ensures that $3\alpha_{1} + 3\alpha_{2} \in W_{6}$ and that it satisfies properties (i) and (ii).
\end{proof}

We now state the auxiliary lemma, which plays a crucial role in the proofs of the principal result of this section.

\begin{lem}\label{lem4.4} Let $3\alpha_1 + \alpha_2$ and $\alpha_1 + 3\alpha_2$ are not in $W$. Then the following properties hold:
\begin{itemize}
    \item[(i)] For any $\mu \in W$,  $\mathrm{diff}(\mu) \leq 1$.
    \item[(ii)] For any weights $\tau_1, \tau_2 \in W$ the Lie bracket $[e_{\tau_1}, e_{\tau_2}]$ vanishes if and only if $\mathrm{diff}(\tau_1) + \mathrm{diff}(\tau_2) = 0$.
\end{itemize}
\end{lem}
\begin{proof} We proceed the proof by induction on index $s$ in $V_s$.  For any $\mu \in V_s$ with $s \leq 5$, the proofs (i) and (ii) follows from Lemma \ref{lem4.3}. 

Assuming that properties (i) and (ii) hold for all weights in $W_s$ and prove them for all weights in $W_{s+1}$. We distinguish odd and even values of $s+1.$ 

\textbf{Case  $s+1=2t+1$}. 
Taking in mind induction assumption, we need to prove lemma for the weights $\mu$ in $V_{2t+1}$. Obviously, we have 
\begin{equation}\label{eq8}
V_{2t+1}\subseteq\{(t+1+i)\alpha_1+(t-i)\alpha_2,\ \ (t-i)\alpha_1+(t+1+i)\alpha_2 \ | \ 0\leq i \leq t-1\}.
\end{equation}

\textbf{Part (i)}: Assume the contrary, that is, that there exists $\mu\in V_{2t+1}$ with $\mathrm{diff}(\mu)> 1$. Then due to \eqref{eq8} we conclude $\mathrm{diff}(\mu) \geq 3$. Consider a  decomposition $\mu=\mu_1+\mu_2$ with $\mu_1, \mu_2\in W$. Then $\mathrm{length}(\mu_i)\leq 2t$ for $i=1,2$. 

By induction assumption it follows that $\mathrm{diff}(\mu_i)\leq 1$ for $i=1,2.$ Then inequality $\mathrm{diff}(\mu)\leq \mathrm{diff}(\mu_1)+\mathrm{diff}(\mu_2)$ leads to the contradiction.

\textbf{Part (ii)}: From Part (i) for $s+1=2t+1$ we conclude
\begin{equation}\label{eq9}
V_{2t+1}\subseteq\{(t+1)\alpha_1+t\alpha_2,\ \ t\alpha_1+(t+1)\alpha_2\}.
\end{equation}

The embedding \eqref{eq9} deduces that $\mathrm{diff}(\tau_1)+\mathrm{diff}(\tau_2)\neq0$, if $\tau_1+\tau_2\in V_{2t+1}$. 

Since $V_{2t+1}\neq\emptyset$, at least one of  structure constants $\{c_{p,p-1}^{t-p+1,t-p+1},\ c_{p-1,p}^{t-p+1,t-p+1}|\ 1\leq p \leq t\}$ must be non-zero.

By considering the equalities:
$$\left\{\begin{array}{llll}
L\left(\alpha_1,p\alpha_1+(p-1)\alpha_2,(t-p)\alpha_1+(t-p+1)\alpha_2\right)=0,\\[3mm]
L\left(\alpha_2,(p-1)\alpha_1+p\alpha_2,(t-p+1)\alpha_1+(t-p)\alpha_2\right)=0,\end{array}\right.$$
for $2\le p\le t$, we get 
\begin{equation} \label{eq10} \left \{ \begin{array}{lll}
c_{t-p,t-p+1}^{1,0}c_{t-p+1,t-p+1}^{p,p-1}=c^{p,p-1}_{t-p,t-p+1}c_{t,t}^{1,0}, \\[3mm]
c_{t-p+1,t-p}^{0,1}c_{t-p+1,t-p+1}^{p-1,p}=c^{p-1,p}_{t-p+1,t-p}c_{t,t}^{0,1},
\end{array}\right.
\end{equation}

The coefficients $c_{t-p,t-p+1}^{1,0}$, $c^{p,p-1}_{t-p,t-p+1}$ $c_{t-p+1,t-p}^{0,1}$ and $c^{p-1,p}_{t-p+1,t-p}$ are all non-zero due to the inductive hypothesis. 

Therefore, from \eqref{eq10} we conclude that $(t+1)\alpha_1+t\alpha_2\in V_{2t+1}$ if and only if $c_{t-p+1,t-p+1}^{p,p-1}\neq0$ for all $1\le p\le t.$ Due to symmetricity between $\alpha_1$ and $\alpha_2$ we get that $t\alpha_1+(t+1)\alpha_2\in V_{2t+1}$ if and only if $c_{t-p+1,t-p+1}^{p-1,p}\neq0$ for all $1\le p\le t.$ 
Thus, we have proved that $[e_{\tau_1}, e_{\tau_2}] \neq 0$ if and only if $\tau_1+\tau_2\in V_{2t+1}$, completing the proof of (ii) for $s= 2t+1$.

\textbf{Case  $s+1=2t+2$}. Due to induction  hypothesis and previous case, we deduce 
$$V_{2t+2} \subseteq \{(t+1+i)\alpha_1+(t+1-i)\alpha_2,\ \ (t+1-i)\alpha_1+(t+1+i)\alpha_2 \ | \ 0\leq i \leq t\}.$$

\textbf{Part (i)}: Suppose \(\mu \in V_{2t+2}\) with \(\mathrm{diff}(\mu) \geq 2\). First, we assume that $\mathrm{diff}(\mu) > 2$, that is, $\mathrm{diff}(\mu) \geq 4$. For a decomposition \(\mu = \mu_1 + \mu_2\), we derive $\mathrm{length}(\mu_i) \leq 2t+1$, i.e., $\mu_1, \mu_2 \in W_{2t+1}$. Applying induction hypothesis and previous case, we get $\mathrm{diff}(\mu_i) \leq 1$. The inequality 
$\mathrm{diff}(\mu)\leq \mathrm{diff}(\mu_1)+\mathrm{diff}(\mu_2)$ implies that $\mathrm{diff}(\mu) > 2$ is not possible. 

Let $\mathrm{diff}(\mu) = 2$. Then we have 
\begin{equation}\label{eq11}
V_{2t+2} \subseteq \{ (t+2)\alpha_1+t\alpha_2,\ \ 
t\alpha_1+(t+2)\alpha_2,\ \ 
(t+1)\alpha_1+(t+1)\alpha_2\}.
\end{equation}

Suppose $\mu = (t+2)\alpha_1 + t\alpha_2 \in V_{2t+2}$. Note that the element $e_{(t+2)\alpha_1+t\alpha_2}$ can be obtained only by products of the form:
\begin{equation}\label{eq12}
[e_{p\alpha_1+(p-1)\alpha_2},e_{(t+2-p)\alpha_1+(t+1-p)\alpha_2}]=c_{p,p-1}^{t+2-p,t+1-p}e_{(t+2)\alpha_1+t\alpha_2}, \ \ 1\leq p \leq t.
\end{equation}

By considering 
$$L(\alpha_1,p\alpha_1+(p-1)\alpha_2,(t-p+1) \alpha_1+(t-p+1)\alpha_2)=0, \quad 2\le p\le t,$$
we obtain the following relations
$$c_{t+1-p,t+1-p}^{1,0}
c_{t+2-p,t+1-p}^{p,p-1}=c^{p,p-1}_{t+1-p,t+1-p}
c_{t+1,t}^{1,0}.$$
From the induction assumption we have $c_{t+1-p,t+1-p}^{1,0}\neq0$ and $c^{p,p-1}_{t+1-p,t+1-p}\neq0$.  It follows that $\mu\in V_{2t+2}$ if and only if $c_{t+2-p,t+1-p}^{p,p-1}\neq 0$ for all $1\leq p \leq t.$ 

Applying for the equalities
$$\left\{\begin{array}{llll}
L(\alpha_1,\alpha_1+\alpha_2,t\alpha_1+(t-1)\alpha_2))&=&0,\\[3mm]
L(\alpha_1,2\alpha_1+\alpha_2,(t-1)\alpha_1+(t-1)\alpha_2)&=&0,
\end{array}\right.$$ 
the same arguments as in Part (ii) for $s+1=2t+1$, one can obtain $\mu \notin W$. By symmetry reason, we  exclude $\mu=t\alpha_1 + (t+2)\alpha_2$. 
Therefore, $V_{2t+2} \subseteq  \{(t+1)\alpha_1 + (t+1)\alpha_2\}$, and (i) holds for $W_{2t+2}$. 

\textbf{Part (ii)}: Let $\tau_1+\tau_2\in V_{2t+2}$ and $\mathrm{diff}(\tau_1)+\mathrm{diff}(\tau_2)=0.$ We fix $k$ from interval $1 \leq k < [\frac{t+1}{2}].$

Note that weights 
$$(t-1)\alpha_1+(t+1)\alpha_2, \ (t-2k+3)\alpha_1+(t-2k+1)\alpha_2, \ (t-2i-1)\alpha_1+(t-2i+1)\alpha_2$$
have $\mathrm{diff}$ more than $1$ and $\mathrm{length}$ less than $2t+1$. Hence, by induction assumption they are not in $W$. Therefore, considering the equalities
$$ L\big(a\alpha_1+b\alpha_2,\ c\alpha_1+d\alpha_2,\ f\alpha_1+g\alpha_2\big)=0$$
for the following choices of triples $\{(a,b),(c,d),(f,g)\}$:

$$\begin{array}{cccc}
\{(1,0),\ (k-1,k),\ (t-k,\ t+1-k)\},\\[3mm]
\{(i,i-1),\ (k-i,k-i+1),\ (t-2k+i+1,\ t-2k+i)\}\\[3mm]
\{(i+1,i),\ (k-i-1,k-i),\ (t-k-i,\ t-k-i+1)\},\\[3mm]
\{(k,k),\ (j,j-1),\ (t-2k-j+1,\ t-2k-j+2)\},
\end{array}$$
with $1\le i\le k-1$ and $1\le j\le t+1-2k$,
we obtain the system of equalities 

\begin{equation}\label{eq13}
\left\{\begin{array}{llllllll}
c_{k,k}^{\,t-k,t-k+1}
&=& -\dfrac{c_{t-k,t-k+1}^{\,1,0}\, c_{t-k+1,t-k+1}^{\,k-1,k}}{c_{1,0}^{\,k-1,k}}, \\[3mm]
c_{k,k}^{\,t-2k+i+1,t-2k+i}
&=& -\dfrac{c_{k-i,k-i+1}^{\,t-2k+i+1,t-2k+i}\, c_{t-k+1,t-k+1}^{\,i,i-1}}
{c_{i,i-1}^{\,k-i,k-i+1}}, &\qquad 1\leq i \leq k-1, \\[3mm]
c_{k,k}^{\,t-k-i,t-k-i+1}&=& -\dfrac{c_{t-k-i,t-k-i+1}^{\,i+1,i}\, c_{t-k+1,t-k+1}^{\,k-i-1,k-i}}{c_{i+1,i}^{\,k-i-1,k-i}}, 
&\qquad 1\leq i \leq k-1, \\[3mm]
c_{k,k}^{\,t-2k-i+1,t-2k-i+2}&=& \dfrac{c_{k,k}^{\,i,i-1}\, c_{k+i,k+i-1}^{\,t-2k-i+1,t-2k-i+2}}{c_{t-k-i+1,t-k-i+2}^{\,i,i-1}}, 
&\qquad 1\leq i \leq t-2k+1.
\end{array}\right.
\end{equation}

To simplify the expressions arising in the subsequent computations, we introduce the following notations:
\[\begin{array}{lllllll}
U_{s,r} &:=& s\, c_{t-k+1,t-k+1}^{\,1,0}\, c_{t-k+2,t-k+1}^{\,k-1,k} 
+ r\, c_{k-1,k}^{\,t-k+1,t-k+1}\, c_{t,t+1}^{\,1,0}, & s, r \in \mathbb{N},\\[3mm] 
C_{p,q} &:=& \dfrac{c_{p,p-1}^{\,t-k-p+1,t-k-p+2}\, c_{t-k+1,t-k+1}^{\,q,q-1}}
{c_{q,q-1}^{\,k-q,k-q+1}}\, c_{t-k+q+1,t-k+q}^{\,k-q,k-q+1}.
\end{array}\]

Next, from the system of equalities valid for $1 \le i \le k$ and $1 \le j \le t+1-k$,
$$\left\{
\begin{array}{l}
L\!\big((t-k+1)\alpha_1+(t-k+1)\alpha_2,\; i\alpha_1+(i-1)\alpha_2,\; (k-i)\alpha_1+(k+1-i)\alpha_2\big)=0,\\[2mm]
L\!\big(k\alpha_1+k\alpha_2,\; j\alpha_1+(j-1)\alpha_2,\; (t-k+1-j)\alpha_1+(t-k+2-j)\alpha_2\big)=0,
\end{array}
\right.$$
we obtain, for $2 \le i \le k$,
\begin{equation}\label{eq14}
\left\{
\begin{aligned}
&U_{1,1}+c_{1,0}^{\,k-1,k}\, c_{k,k}^{\,t-k+1,t-k+1}=0, \\[5mm]
&U_{1,1}+\frac{c_{1,0}^{\,k-1,k}}{c_{i,i-1}^{\,k-i,k-i+1}}\!\left(c_{t-k+1,t-k+1}^{\,i,i-1}\, c_{t-k+i+1,t-k+i}^{\,k-i,k-i+1}
+c_{k-i,k-i+1}^{\,t-k+1,t-k+1}\, c_{t-i+1,t-i+2}^{\,i,i-1}
\right)=0, \\[5mm]
&U_{1,1}+\frac{c_{1,0}^{\,k-1,k}}{c_{j,j-1}^{\,t-k-j+1,t-k-j+2}}\!\left(c_{k,k}^{\,j,j-1}\, c_{k+j,k+j-1}^{\,t-k-j+1,t-k-j+2}
+c_{t-k-j+1,t-k-j+2}^{\,k,k}\, c_{t-j+1,t-j+2}^{\,j,j-1}
\right)=0.
\end{aligned}
\right.
\end{equation}

For the sake of convenience we 
denote the first, the second and the third types of equalities of \eqref{eq14} by $A_1$, $A_i$ and $B_j$, respectively. 

Next we are going to analyze the system   \eqref{eq14} by the following steps:

{\bf Step 1}: Substituting the coefficient $c_{k,k}^{t-k+1,t-k+1}$ from $A_1$ into $B_1$, we obtain
\begin{equation}\label{eq15}
c_{k,k}^{1,0}c_{k+1,k}^{t-k,t-k+1}c_{t-k+1,t-k+1}^{k-1,k}+c_{k,k}^{t-k,t-k+1}U_{1,2}=0.
\end{equation}

Then, using \eqref{eq15} as the inductive base in the variable~$j$, and invoking the fourth equality in~\eqref{eq13} at each inductive step, we obtain the following general expression for $B_{jk+1}$, valid for all
$0 \le j \le \frac{t+1-i}{k}:$
\begin{equation}\label{eq16} 
c^{(j+1)k+1,(j+1)k}_{t-(j+1)k,t-(j+1)k+1}=\frac{c_{jk+1,jk}^{t-(j+1)k,t-(j+1)k+1}}{c_{1,0}^{k-1,k}c_{k,k}^{jk+1,jk}}U_{j+1,j+2}.
\end{equation}

{\bf{Step 2}}: Extracting the coefficient $c_{t-i+1,\, t-i+2}^{\,i,\, i-1}$ from equality $A_i$ for all $2 \le i \le k$, we substitute it into corresponding $B_i$ together with the third equality in~\eqref{eq13}. The we obtain equality of $B_i$ in the form

\begin{equation}\label{eq17}
c^{k+i,k+i-1}_{t-k-i+1, t-k-i+2}
=\frac{2 c_{i,i-1}^{t-k-i+1,t-k-i+2} U_{1,1}- C_{i,i}c_{1,0}^{k-1,k}}{c_{1,0}^{k-1,k}c_{k,k}^{i,i-1}}.
\end{equation}

Using \eqref{eq17} as the inductive base in the variable $j$, and applying the fourth equality in~\eqref{eq13} at each step of induction, we derive the form of $B_{jk+i}$, valid for all $2 \le i \le k$ and $0 \le j \le \frac{t+1-i}{k}$:
\begin{equation}\label{eq18}
c_{t-(j+1)k-i+1,t-(j+1)k-i+2}^{\, (j+1)k+i,(j+1)k+i-1} 
= \frac{(j+2)\, c_{jk+i,jk+i-1}^{\, t-(j+1)k-i+1,t-(j+1)k-i+2} \, U_{1,1} 
        - C_{jk+i,i}\, c_{1,0}^{\, k-1,k}}{c_{1,0}^{\, k-1,k} \, c_{k,k}^{\, jk+i,jk+i-1}}.
\end{equation}

{\bf Step 3}: Before start this step we re-numerate the last $k$ equalities in \eqref{eq14} by supposing $B_{t-2k+1+p}$ with $1\leq p \leq k$ as 
$B_{jk+i_0+p}$ for some fixed $1 \le i_0 \le k$ and $j=\frac{t-2k-i_0+1}{k}.$ Moreover, we define $s_l$ such that
$$\left\lfloor \frac{lk - 1}{i_0} \right\rfloor < s_l \le \left\lfloor \frac{lk - 1}{i_0} \right\rfloor + 1, \quad 1 \le l \le r,$$
where $r = \frac{i_0}{\gcd(k, i_0)}.$

Substituting the second equality of \eqref{eq13} into \eqref{eq18} for $B_{jk+i_0+1}$, we obtain the base of induction 
providing by $l$ for $1\le l\le s_1-1$ to establish the expression for  $B_{jk+li_0+1}$ in the following form:

\begin{equation}\label{eq19}c_{t-k+li_0+2, t-k+li_0+1}^{k-li_0-1, k-li_0}=\frac{c_{li_0+1,li_0}^{k-li_0-1, k-li_0}}{c_{1,0}^{k-1,k}c_{t-k+1,t-k+1}^{li_0+1, li_0}}U_{l(j+2)+1,l(j+2)}.
\end{equation}

Assume that \eqref{eq19} holds for all $1\le l\le s_1-2$. Consider the expression for $B_{jk+(l+1)i_0+1}$ given in \eqref{eq18}. A direct substitution of the term $c_{t-k+li_0+2, t-k+li_0+1}^{k-li_0-1, k-li_0}$ using  \eqref{eq19} and replacing the $c_{k,k}^{t-2k+li_0+2,t-2k+li_0+1}$ via the second equality in \eqref{eq13}, yields the form \eqref{eq22} for $B_{j_0k+(l+1)i_0+1}$. 
 
Now, consider the expression$B_{j_0k+s_1i_0+1}$ from \eqref{eq18}. Substituting $c_{t-k+(s_1-1)i_0+2, t-k+(s_1-1)i_0+1}^{k-(s_1-1)i_0-1, k-(s_1-1)i_0}$ with the right hand side of \eqref{eq22} for the case $l=s_1-1$ and again applying the second equality of \eqref{eq13} to the structure constant $c_{k,k}^{t-2k+(s_1-1)i_0+2, t-2k+(s_1-1)i_0+1},$ we obtain the following expression for $B_{j_0k+s_1i_0+1}$:
\begin{equation}\label{eq20}
c_{t-2k+s_1i_0+2,t-2k+s_1i_0+1}^{2k-s_1i_0-1,2k-s_1i_0}=\frac{c_{s_1i_0-k+1,s_1i_0-k}^{2k-s_1i_0-1,2k-s_1i_0}}{c_{1,0}^{k-1,k}c_{t-k+1,t-k+1}^{s_1i_0-k+1,s_1i_0-k}}U_{s_1(j_0+2)+2,s_1(j_0+2)+1}
\end{equation}

Next, we analyze $B_{j_0k+(s_1+1)i_0+1-k}$ using \eqref{eq18}. Replacing the coefficient 
$c_{t-2k+s_1i_0+2,\; t-2k+s_1i_0+1}^{\,2k-s_1i_0-1,\;2k-s_1i_0}$ with the right-hand side of \eqref{eq20}, and rewriting $c_{k,k}^{\,j_0k+(s_1+1)i_0-k+1,\; j_0k+(s_1+1)i_0-k}$, via the second equality of \eqref{eq13} gives $B_{j_0k+(s_1+1)i_0+1-k}$:
\begin{equation}\label{eq21}
\begin{aligned}
c_{t-2k+(s_1+1)i_0+2,\; t-2k+(s_1+1)i_0+1}^{\,2k-(s_1+1)i_0-1,\; 2k-(s_1+1)i_0}
&=\frac{c_{(s_1+1)i_0-k+1,\; (s_1+1)i_0-k}^{\,2k-(s_1+1)i_0-1,\; 2k-(s_1+1)i_0}
}{c_{1,0}^{k-1,k}\,c_{t-k+1,\; t-k+1}^{(s_1+1)i_0-k+1,\; (s_1+1)i_0-k}}\\[2mm]
&\quad \times  U_{(s_1+1)(j_0+2)+2,\; (s_1+1)(j_0+2)+1}. 
\end{aligned}
\end{equation}

We now carry out an induction on $s_q$ for $1\le q\le r-1$.  Using \eqref{eq19} and \eqref{eq20} as the base of this induction, we establish the explicit forms of the equalities $B_{j_0 k + (s_q + u)i_0 + 1 - qk}$ with 
$1 \le u \le s_{q+1} - s_q - 1$ 
and for the case $q\neq r-1$  of $B_{(j_0+1)k+s_{q+1}i_0+1-(q+1)k}$. More precisely, we obtain
\begin{equation}\label{eq22}
\begin{aligned}
c_{t-(q+1)k+(s_q+u)i_0+2,\; t-(q+1)k+(s_q+u)i_0+1}^{(q+1)k-(s_q+u)i_0-1,\; (q+1)k-(s_q+u)i_0}&=\frac{c_{(s_q+u)i_0-qk+1,\; (s_q+u)i_0-qk}^{(q+1)k-(s_q+u)i_0-1,\; (q+1)k-(s_q+u)i_0}}{c_{1,0}^{k-1,k}\, c_{t-k+1,\; t-k+1}^{(s_q+u)i_0-qk+1,\; (s_q+u)i_0-qk}}\\[2mm]
&\quad\times U_{(s_q+u)(j_0+2)+q,\; (s_q+u)(j_0+2)+q-1},
\end{aligned}
\end{equation}
and 
\begin{equation}\label{eq23}
\begin{aligned}
c_{t-(q+2)k+s_{q+1}i_0+2,\; t-q+2)k+s_{q+1}i_0+1}^{(q+2)k-s_{q+1}i_0-1,\; (q+2)k-s_{q+1}i_0}&=\frac{c_{s_{q+1}i_0-(q+1)k+1,\; s_{q+1}i_0-(q+1)k}^{(q+2)k-s_{q+1}i_0-1,\; (q+2)k-s_{q+1}i_0}}{c_{1,0}^{k-1,k}\,c_{t-k+1,\; 
t-k+1}^{s_{q+1}i_0-(q+1)k+1,\; s_{q+1}i_0-(q+1)k}}\\[2mm]
&\quad\times U_{s_{q+1}(j_0+2)+q+1,\; s_{q+1}(j_0+2)+q}.
\end{aligned}
\end{equation}

Let assume for all  $1\le u\le s_{q+1}-s_q-2$ the equality $B_{j_0k+(s_q+u)i_0+1-qk}$ can be written as \eqref{eq22}. 

Next, we consider the equality $B_{j_0k+(s_{q+1}+1)i_0+1-(q+1)k}$ by means of \eqref{eq18}. Substituting the structure constants
$$c_{t-(q+2)k+s_{q+1}i_0+2,\; t-(q+2)k+s_{q+1}i_0+1}^{(q+2)k-s_{q+1}i_0-1,\; (q+2)k-s_{q+1}i_0}\quad\mbox{and}\quad c_{k,k}^{t-(q+3)k+s_{q+1}i_0+2,t-(q+3)k+s_{q+1}i_0+1},$$
by the right-hand side of \eqref{eq23} and the second equality in \eqref{eq13}, respectively, we obtain the equality 
$$B_{j_0k+(s_{q+1}+1)i_0+1-(q+1)k}$$
whose resulting expression is of the same form as in \eqref{eq22}.

Assume for all $1\le u\le s_{q+2}-s_{q+1}-2$, the equality $B_{j_0k+(s_{q+1}+u)i_0+1-(q+1)k}$ has the form as \eqref{eq22}.

Let consider the form of $B_{j_0k+(s_{q+1}+u+1)i_0+1-(q+1)k}$ given by \eqref{eq18}. Then taking instead of the following structure constants
$$ c_{t-(q+2)k+(s_{q+1}+u)i_0+2, t-(q+2)k+(s_{q+1}+u)i_0+1}^{(q+2)k-(s_{q+1}+u)i_0-1, (q+2)k-(s_{q+1}+u)i_0}\quad\mbox{and}\quad c_{k,k}^{t-(q+3)k+(s_{q+1}+u)i_0+2,t-(q+3)k+(s_{q+1}+u)i_0+1},$$
respectively, with the right hand side of \eqref{eq22} and the second equality of \eqref{eq13}, we get the form of $$B_{j_0k+(s_{q+1}+u+1)i_0+1-(q+1)k}$$ as \eqref{eq22}. 

Next by putting the expression of structure constants 
$$c_{t-(q+2)k+(s_{q+2}-1)i_0+2, t-(q+2)k+(s_{q+2}-1)i_0+1}^{(q+2)k-(s_{q+2}-1)i_0-1, (q+2)k-(s_{q+2}-1)i_0}\quad \mbox{and}\quad c_{k,k}^{t-(q+3)k+(s_{q+2}-1)i_0+2,t-(q+2)k+(s_{q+2}-1)i_0+1},$$ respectively, from \eqref{eq22} and the second equality of \eqref{eq13} into $B_{(j_0+1)k+s_{q+2}i_0-(q+2)k+1}$ given the form by \eqref{eq18}, we immediately obtain the form of $B_{(j_0+1)k+s_{q+2}i_0-(q+2)k+1}$ as \eqref{eq23}. Thus, we have proved \eqref{eq22} and \eqref{eq23}.

Finally, we analyze the equality $B_{(j_0+1)k+1}$ in \eqref{eq16}. Substituting the structure constant $c_{t-i_0+2,t-i_0+1}^{i_0-1,i_0}$ with the right-hand side of equation \eqref{eq22} under the parameter assignments $q = r-1$ and $u = s_r - s_{r-1} - 1$ and further incorporating the expression for $c_{t-k-i_0+2,t-k-i_0+1}^{i_0-1,i_0}$ derived from the relation
\[
L\big( (i_0-1,i_0),\ (k-i_0+1,k-i_0),\ (t-k-i_0+2,t-k-i_0+1) \big) = 0,
\]
we obtain the following form for the equality $B_{(j_0+1)k+1}$:
$$\frac{c_{k-i_0+1,k-i_0}^{i_0-1,i_0}}{c_{t-k+1,t-k+1}^{k-i_0+1,k-i_0}} \times \left( s_r(j_0+2) + r - 1 \right) U_{1,1} = 0.$$

Since $c_{k-i_0+1,k-i_0}^{i_0-1,i_0}$ and $c_{t-k+1,t-k+1}^{k-i_0+1,k-i_0}$ are non-zero, we get $U_{1,1} = 0$.
From the equality $A_1$, we obtain that $c_{k,k}^{t-k+1, t-k+1}=0.$ Therefore, for the arbitrary chosen $\tau_1$ and $\tau_2$ such that $\mathrm{diff}(\tau_1)+\mathrm{diff}(\tau_2)=0$ we conclude $[e_{\tau_1},e_{\tau_2}]=0$.

Let now prove the part ``if''. Assume that $\tau_1,\tau_2\in W$ with $\tau_1+\tau_2\in V_{2t+2}$ such that $\mathrm{diff}(\tau_1)+\mathrm{diff}(\tau_2)\neq0$. Then, without loss of generality, one can assume that weights $\tau_1$ and $\tau_2$ can be written in the following forms:
$$\tau_1=p\alpha_1+(p-1)\alpha_2,\ \tau_2=(t-p+1)\alpha_1+(t-p+2)\alpha_2,\ 1\le p\le t+1.$$

From the equalities 
$$L\big((1,0),(p-1,p-1),(t-p+1,t-p+2)\big)=0, \quad 2\leq p \leq t+1,$$ 
we get
\begin{equation}\label{eq24}
\begin{aligned}
    c_{p,p-1}^{t-p+1, t-p+2}=\frac{c^{p-1,p-1}_{t-p+1,t-p+2}}{c_{1,0}^{p-1,p-1}}c_{t,t+1}^{1,0}.\\[3mm]
\end{aligned}
\end{equation}

By induction assumption for Part (ii), all the coefficients $c_{1,0}^{p-1,p-1}$ and $c^{p-1,p-1}_{t-p+1,t-p+2}$ are non-zero \eqref{eq24} for all $2\leq p \leq t+1$. Consequently, the vanishing of the structure constants $c_{p,p-1}^{t-p+1, t-p+2}$ for some $p$ 
implies the vanishing of all of them. This is a contradiction with assumption $\tau_1+\tau_2\in V_{2t+2}$, which leads that $[e_{\tau_1},e_{\tau_2}]\neq0$. Thus, the (ii) holds for $V_{2t+2}$.
\end{proof}

\begin{rem}\label{rm4.5} From the proof of Lemma \ref{lem4.4} we conclude that each set $W_i, 1\leq i \leq s-1$, where $s$ is nilindex of $\cal N$, consists of all weights $\mu $ with 
$\mathrm{diff}(\mu)\le 1$ and $\mathrm{length}(\mu)\le i.$
\end{rem}

Building on the results of the preceding lemmas, we now establish the following proposition.
\begin{prop}\label{prop4.7}
Let $\mathcal R_{\mathcal T}$ be a solvable Lie algebra with nilradical $\cal N$ and let $\mathcal M$ be an $\mathcal R_{\mathcal T}$-module such that  the action of  $\mathcal T$ on $\mathcal M$ is diagonal, and all the weights of $\mathcal T$ on $\mathcal M$ are in $W$. Assume further that
$3\alpha_1+\alpha_2 \notin W$ and $\alpha_1+3\alpha_2 \notin W.$ Then $H^2(\cal N,\cal M)^{\cal T}=0$.
\end{prop}
\begin{proof} Applying the same arguments in the proof of Theorem \ref{thm3.2}, it is enough to consider $\dim \cal M = 1$. 
Suppose $\cal M=\mathbb Cm$ and the weight of $\cal T$ on $\cal M$ is $\Lambda$. Consider an arbitrary $f\in H^2(\cal N,\cal M)^{\cal T}$ such that $\cal T\cdot f=0$. 

Due to Lemma \ref{lem4.4}, all weights in $ W $ have one of the following forms:
$$n\alpha_1 + n\alpha_2, \quad (n+1)\alpha_1 + n\alpha_2, \quad n\alpha_1 + (n+1)\alpha_2. $$

Thank to symmetry the case of $\alpha_1$ and $\alpha_2$ 
it is enough to consider the cases $n\alpha_1 + n\alpha_2$ and $(n+1)\alpha_1 + n\alpha_2.$

Firstly, we assume that $\Lambda=n\alpha_1 + n\alpha_2$. 
From equalities with $1\leq q \leq n-1$:
$$\left \{ \begin{array}{lll}
L\big(\alpha_1,\ (q-1)\alpha_1+q\alpha_2,\ (n-q)\alpha_1+(n-q)\alpha_2\big)=0, \\[3mm]
Z \big(\alpha_1,\ (q-1)\alpha_1+q\alpha_2,\ (n-q)\alpha_1+(n-q)\alpha_2\big)=0,
\end{array}\right. $$
we obtain the restrictions on structure constants
\begin{equation}\label{eq25} \left \{ \begin{array}{lll}
c_{n-q,n-q}^{1,0}c_{n-q+1,n-q}^{q-1,q}+
c^{q-1,q}_{n-q,n-q}c^{n-1,n}_{1,0} &=&0, \\[3mm]
c_{n-q,n-q}^{1,0}\Phi_{n-q+1,n-q}^{q-1,q}+
c^{q-1,q}_{n-q,n-q}\Phi^{n-1,n}_{1,0} &=&0.
\end{array}\right.
\end{equation}

Due to Lemma \eqref{lem4.4}, we have $c_{n-q+1,n-q}^{q-1,q} \neq 0$ and $c^{q-1,q}_{n-q,n-q}\neq 0$. Therefore, by using  \eqref{eq25}, we get 
\begin{equation}\label{eq26}
\Phi_{n-q+1,n-q}^{q-1,q}=
\frac{c_{n-q+1,n-q}^{q-1,q}}{c^{n-1,n}_{1,0}}\Phi^{n-1,n}_{1,0},\quad 1\leq q \leq n-1.\end{equation}

Let define cochain $g \in C^1(\cal N, \cal R)$ as follows
$$g(e_{\rho})=0 \quad \text{for} \quad \rho \neq \Lambda=n\alpha_1+n\alpha_2\quad \text{and}\quad 
g(e_{\Lambda})=-\frac{\Phi_{1,0}^{n-1,n}}{c_{1,0}^{n-1,n}}m.$$

Assume that $\alpha+\beta=\Lambda$ and $\alpha$ is either $\alpha_1$ or $\alpha_2$.

If $\alpha=\alpha_1$, then from definition of $g$ we immediately get $(f-d^1{g})(e_{\alpha_1}, e_{\beta})=0.$

If $\alpha=\alpha_2$, then applying the equality \eqref{eq26} for $q=1$, we obtain $(f-d^1g)(e_{\alpha_2}, e_{\beta})=0.$

Assume $\alpha$ and $\beta$ have length more than 1. Without loss of generality, we can assume 
$$\alpha=(q-1)\alpha_1+q\alpha_2, \quad \beta=(n-q+1)\alpha_1+(n-q)\alpha_2 \quad \mbox{for some} \quad 2 \leq q \leq n-1.$$ 
Then by \eqref{eq26} for we get, 
$(f-d^1{g})(e_\alpha, e_{\beta}) = \Phi_{q-1,q}^{n-q+1,n-q}m+c_{q-1,q}^{n-q+1,n-q}g(e_{\Lambda})=0.$ Hence, $f=d^1 g.$

Secondly, assume that $\Lambda=(n+1)\alpha_1 + n\alpha_2 \in W$. From the equalities
$$\begin{array}{lll}
L\big(\alpha_1,\ q\alpha_1+(q-1)\alpha_2,\
(n-q)\alpha_1+(n-q+1)\alpha_2\big)=0,  \\[3mm]
Z\big(\alpha_1,\ q\alpha_1+(q-1)\alpha_2,\
(n-q)\alpha_1+(n-q+1)\alpha_2\big)=0,
\end{array}$$
with $2\leq q \leq n-1$, we deduce
\begin{equation}\label{eq27}
\left\{\begin{array}{lll}
c_{n-q,n-q+1}^{1,0}c_{n-q+1,n-q+1}^{q,q-1}+c_{q,q-1}^{n-q,n-q+1}c_{n,n}^{1,0}&=&0,\\[3mm] 
c_{n-q,n-q+1}^{1,0}\Phi_{n-q+1,n-q+1}^{q,q-1}+c_{q,q-1}^{n-q,n-q+1}\Phi_{n,n}^{1,0}&=&0.
\end{array}\right.
\end{equation}

According Lemma \eqref{lem4.4} we have $c_{n-q,n-q+1}^{1,0} \neq 0$ and $c^{q,q-1}_{n-q,n-q+1} \neq 0$. Then the system of equalities \eqref{eq27} implies 

\begin{equation}\label{eq28}
\Phi_{n-q+1,n-q+1}^{q,q-1}=\frac{c_{n-q+1,n-q+1}^{q,q-1}}{c_{n,n}^{1,0}}\Phi_{n,n}^{1,0}, \quad 2\leq q \leq n-1.  
\end{equation}
Define $g \in C^1(\cal N, \cal R)$ as
$$g(e_{\rho})=0 \quad \text{for} \quad \rho \neq \Lambda=(n+1)\alpha_1+n\alpha_2 \quad \mbox{and} \quad 
g(e_{\Lambda})=-\frac{\Phi^{n,n}_{1,0}}{c^{n,n}_{1,0}}m.$$

Now we consider the case of $\alpha+\beta=\Lambda$ with $\alpha$ has length 1. Then the case of $\alpha=\alpha_2$ is not possible due to $\mathrm{diff}(\Lambda-\alpha_2)=2$, which contradicts to assertion of Lemma \ref{lem4.4}. Hence, we have $\alpha=\alpha_1$. In that case, 
$(f-d^1g)(e_{\alpha_1}, e_{\beta})=0.$

Let $\alpha+\beta=\Lambda$ such that $\alpha$ and $\beta$ have length more than 1. Without loss of generality, we can assume
$$\alpha=q\alpha_1+(q-1)\alpha_2,\ \ \ 
\beta=(n-q+1)\alpha_1+(n-q+1)\alpha_2, \quad  2 \leq q \leq n-1.$$

Then by using \eqref{eq28}, we can obtain $(f-d^1g)(e_\alpha, e_{\beta}) =0.$ Therefore, $f=d^1g.$
\end{proof}

In the following theorem we give triviality of the second cohomology groups of $\cal R_{\cal T}$ with adjoint module.

\begin{thm}\label{thm4.8}
Let $\mathcal R_{\mathcal T}$ be a solvable Lie algebra with nilradical $\mathcal N$. Assume further that $3\alpha_1+\alpha_2 \notin W$ and $\alpha_1+3\alpha_2 \notin W.$ Then $\mathcal R_{\mathcal T}$ is cohomologically rigid.
\end{thm}
\begin{proof} Proposition \ref{prop4.7} and completeness of the algebra $\cal R_{\cal T}$ imply that 
\(H^{i}(\mathcal{N}, \mathcal{M})^{\mathcal{T}} = 0\) for all \(0 \leq i \leq 2\). Now, the claim follows directly from  Remark~\ref{rm2.2}.
\end{proof}

In fact, the conditions in Theorem \ref{thm4.8} ensuring cohomological rigidity of the algebra $\mathcal{R}_{\mathcal{T}}$ are not necessary conditions. This is demonstrated by the model nilpotent and model filiform algebras, which remain cohomologically rigid despite satisfying $\alpha_1 + 3\alpha_2 \in W$ (see, Proposition 4 in \cite{Ancochea4} and \cite{libroKluwer}).

Below we present one-parametric family of nilpotent Lie algebras of maximal rank that satisfies to assumptions of Theorem~\ref{thm4.8}.
\begin{exa}\label{exa3.7} For each integer \(n \ge 1\) and each parameter $t \in \mathbb{C}$, consider the \(3(n+1)\)-dimensional nilpotent Lie algebra \(\mathfrak{g}_{n,t}\) with basis
\[\big\{
e_{\,k\alpha_1 + (k-1)\alpha_2},\;
e_{\, (k-1)\alpha_1 + k\alpha_2},\;
e_{\, k(\alpha_1 + \alpha_2)}
\;\big|\; 1 \le k \le n+1
\big\},\]
and nonzero brackets given by
$$\left\{\begin{array}{rlll}
[e_{(i-1)\alpha_1 + i\alpha_2},\, e_{(j-i+1)\alpha_1 + (j-i)\alpha_2}]&= (-1)^{\delta_{i,1}}\,e_{j\alpha_1 + j\alpha_2}, \\[3mm]
[e_{i\alpha_1 + i\alpha_2},\, e_{(j-i+1)\alpha_1 + (j-i)\alpha_2}]&= (-1)^{\delta_{i,j}}\,e_{(j+1)\alpha_1 + j\alpha_2}, \\[3mm]
[e_{i\alpha_1 + i\alpha_2},\, e_{(j-i)\alpha_1 + (j-i+1)\alpha_2}]&= -(-1)^{\delta_{i,j}}\,
e_{j\alpha_1 + (j+1)\alpha_2}, \\[3mm]
[e_{\alpha_1},\, e_{n\alpha_1 + (n+1)\alpha_2}]
&= t\, e_{(n+1)\alpha_1 + (n+1)\alpha_2}, \\[3mm]
[e_{(i-1)\alpha_1 + i\alpha_2},\, e_{(n-i+2)\alpha_1 + (n-i+1)\alpha_2}]&= t\, e_{(n+1)\alpha_1 + (n+1)\alpha_2},\\[3mm]
\end{array}
\right.$$
where \(\delta_{i,j}\) denotes the Kronecker delta and \(1 \le i \le j \le n\).

The algebra \(\mathfrak{g}_{n,t}\) admits a two-dimensional torus of maximal rank whose associated root system $W$ contains neither the weight $3\alpha_1 + \alpha_2$ nor the weight $\alpha_1 + 3\alpha_2$. It therefore follows from Theorem~\ref{thm4.8}, that the maximal solvable extension of $\mathfrak{g}_{n,t}$ is cohomologically rigid.
\end{exa}

\section{Non-cohomologically rigidity of solvable Lie algebras of maximal rank}\label{sec5}

In contrast with the previous two sections, which established rigidity phenomena by proving the vanishing of $H^{2}(\mathcal{R}_{\mathcal{T}},\mathcal{R}_{\mathcal{T}})$ and of certain higher cohomology groups, this section is devoted to solvable Lie algebras $\mathcal{R}_{\mathcal{T}}$ for which cohomological rigidity fails. We describe the structural configurations that lead to the non-vanishing of the second cohomology group and present the main results of this section below.

\begin{thm}\label{thm5.1} Suppose $\cal R_{\cal T}$ be a solvable Lie algebra  with nilradical $\cal N$ such that there exist $\alpha, \beta, \gamma \in \Pi$ satisfying one of the following conditions:
$$\begin{array}{llllcc}
i)& e_{3\alpha+2\beta}\in \operatorname{Center}(\cal N)& and &\{3\alpha + \beta,\ 2\alpha + 2\beta\} &\subseteq &W;\\[3mm]
ii)& e_{4\alpha+2\beta}\in \operatorname{Center}(\cal N)& and & \{4\alpha+\beta,\ 3\alpha+ 2\beta, 2\alpha+2\beta\} &\subseteq& W;\\[3mm]
iii)&e_{\alpha+\beta+\gamma}\in \operatorname{Center}(\cal N)&and &\{\alpha + \beta,\ \alpha + \gamma,\ \beta + \gamma\}&\subseteq &W.\\[3mm]
\end{array}
$$
Then $H^2(\cal R_{\cal T}, \cal R_{\cal T}) \neq 0$.
\end{thm}
\begin{proof}
\textbf{Case (i).}
Since $3\alpha+\beta\in W$, we have
$c(\alpha, p\alpha+\beta)\neq 0,\quad 0\le p\le 2.$ From the identity $L(\alpha,\beta,\alpha+\beta)=0$ and the assumption $2\alpha+2\beta\in W$, we obtain
$$c(\beta,2\alpha+\beta)\neq 0,
\qquad c(\alpha,\alpha+2\beta)\neq 0.$$

Furthermore, from the equality $L(\alpha,\beta,2\alpha+\beta)=0$ it follows that at least two of
$$c(\alpha+\beta,2\alpha+\beta),\qquad
c(2\alpha+2\beta,\alpha),\qquad
c(3\alpha+\alpha,\beta)$$
are non-zero.

Let $c(\beta,3\alpha+\beta)\neq 0$. Then at least one of $c(\alpha,2\alpha+2\beta)$ or $c(\alpha+\beta,2\alpha+\beta)$ is non-zero.  
Define $f\in C^{2}(\mathcal{N},\mathcal{R}_{\mathcal{T}})^{\mathcal{T}}$ by
$$\left\{\begin{array}{lllll}
f(e_{\alpha},e_{2\alpha+2\beta})&=&e_{3\alpha+2\beta},\\[2mm]
f(e_{\alpha+\beta},e_{2\alpha+\beta})&=&\frac{c(\beta,2\alpha+\beta)}{c(\alpha,\beta)}\, e_{3\alpha+2\beta},\\[2mm]
\end{array}\right.$$
and $f(e_{\tau_1},e_{\tau_2})=0$ for all the other pairs of weights $\tau_1,\tau_2$.

One verifies directly that $f\in Z^{2}(\mathcal{N},\mathcal{R}_{\mathcal{T}})^{\mathcal{T}}$.  
Assume that $f=d^{1}g$ for some $g\in C^{1}(\mathcal{N},\mathcal{R}_{\mathcal{T}})^{\mathcal{T}}$ with $g(e_{\alpha})=\mu_{\alpha}e_{\alpha}$.  
Evaluating $(f-d^{1}g)(e_{\alpha},e_{\beta})=0$ 
yields the system, which has no a solution. Hence $f\notin B^{2}(\mathcal{N},\mathcal{R}_{\mathcal{T}})^{\mathcal{T}}$.

Let $c(\beta,3\alpha+\beta)=0$. Then both the structure constants $c(\alpha,2\alpha+2\beta)$ and $c(\alpha+\beta,2\alpha+\beta)$ are non-zero.  
Define
$$\begin{cases}
f(e_{\beta},e_{3\alpha+\beta})=e_{3\alpha+2\beta},\\[2mm]
f(e_{\alpha+\beta},e_{2\alpha+\beta})
=\frac{c(2\alpha+\beta,\alpha)}{c(\alpha,\beta)}\,
e_{3\alpha+2\beta},\\[2mm]
\end{cases}$$
and $f(e_{\tau_1},e_{\tau_2})=0$ for all the other pairs of weights $\tau_1,\tau_2$. 
A similar computation shows that again $f\notin B^{2}(\mathcal{N},\mathcal{R}_{\mathcal{T}})^{\mathcal{T}}$, hence $
0\neq f\in H^{2}(\mathcal{N},\mathcal{R}_{\mathcal{T}})^{\mathcal{T}}.$

\textbf{Case (ii)} Since $4\alpha+\beta\in W$, the structure constants $c(\alpha, p\alpha+\beta)$ are non-zero for $ 0\le p\le 3.$ From the equality $L(\alpha,\beta,2\alpha+\beta)=0$ and the condition $3\alpha+2\beta\in W$, it follows that  at least two of the following structure constants are non-zero:
$$
c(\beta,3\alpha+\beta),
\qquad
c(\alpha,2\alpha+2\beta),\quad c(\alpha+\beta,2\alpha+\beta).$$

Similarly, the equalities $L(\alpha,\beta,3\alpha+\beta)=0$  and $L(\alpha,\alpha+\beta,2\alpha+\beta)=0$ imply that at least two of the followings are non-zero:
$$c(\alpha,3\alpha+2\beta),\quad c(\beta,4\alpha+\beta),\quad c(\alpha+\beta,3\alpha+\beta).$$

Define a 2-cocycle $f\in C^{2}(\mathcal{N},\mathcal{R}_{\mathcal{T}})^{\mathcal{T}}$ by
$$
\left\{\begin{array}{llllllll}
f(e_{\beta},e_{3\alpha+\beta})&=&e_{3\alpha+2\beta},&
f(e_{\beta},e_{4\alpha+\beta})&=&\frac{c(3\alpha+2\beta,\alpha)c(\beta,2\alpha+\beta)}{c(2\alpha+\beta,\alpha)c(3\alpha+\beta,\alpha)}e_{4\alpha+2\beta},\\[3mm]
f(e_{\alpha+\beta},e_{2\alpha+\beta})&=&\frac{c(\beta,2\alpha+\beta)}{c(\alpha,\beta)}\, e_{3\alpha+2\beta},&
f(e_{\alpha+\beta},e_{3\alpha+\beta})&=&\frac{c(3\alpha+2\beta,\alpha)c(\beta,2\alpha+\beta)}{c(2\alpha+\beta,\alpha)c(\alpha,\beta)}e_{4\alpha+2\beta},\\[3mm]
\end{array}\right.$$
and $f(e_{\tau},e_{\pi})=0$ on all other pairs of weights.

One can verify that $f\in Z^2(\cal N,\cal R_{\cal T})^{\cal T}$. By assuming $f=d^1g$ for some $g\in C^{1}(\mathcal{N},\mathcal{R}_{\mathcal{T}})^{\mathcal{T}}$ with $g(e_{\gamma})=\mu_{\gamma}e_{\gamma}$, we get equations system which does not have any solution with respect to variables $\mu_{\gamma},\  \gamma\in W.$ Hence, $0\neq f\in H^2(\cal R_{\cal T},\cal R_{\cal T}).$

\textbf{Case (iii).}
The equality $L(\alpha,\beta,\gamma)=0$ implies that at least two of the structure constants
$$c(\alpha,\beta+\gamma),\qquad c(\beta,\alpha+\gamma),\qquad c(\gamma,\alpha+\beta)$$
are non-zero.

If $c(\beta,\alpha+\gamma)\neq 0$  and $c(\gamma,\alpha+\beta)\neq0$, define
$$\left\{\begin{array}{llll}
f(e_{\alpha},e_{\beta+\gamma})&=&e_{\alpha+\beta+\gamma},\\[3mm]
f(e_{\beta},e_{\alpha+\gamma})
&=&\frac{c(\beta,\gamma)}{c(\alpha,\gamma)}\,
e_{\alpha+\beta+\gamma},
\end{array}\right.$$
$f(e_{\tau},e_{\pi})=0$ for all the other pairs of weights.

Assuming again that $f=d^{1}g$ leads to an inconsistent system of equations for the unknown variables  $\mu_{\alpha},$ $\mu_{\beta}$, $\mu_{\gamma}$ and $\mu_{\alpha+\beta+\gamma}$. Thus $f\notin B^{2}(\mathcal{N},\mathcal{R}_{\mathcal{T}})^{\mathcal{T}}$.

Finally, if $c(\beta,\alpha+\gamma)=0$, then the remaining two structure constants are non-zero, and one defines
$$\left\{\begin{array}{llll}
f(e_{\beta},e_{\alpha+\gamma})&=&e_{\alpha+\beta+\gamma},\\[3mm]
f(e_{\alpha+\beta},e_{\gamma})
&=&\frac{c(\gamma,\alpha)}{c(\alpha,\beta)}
\,e_{\alpha+\beta+\gamma},\\[2mm]
\end{array}\right.$$
$f(e_{\tau},e_{\pi})=0$ for all the other pairs of weights.

The same argument yields $0\neq f\in H^{2}(\mathcal{N},\mathcal{R}_{\mathcal{T}})^{\mathcal{T}}.$ This completes the proof of theorem. \end{proof}

\begin{rem} Note that the condition in Theorem \ref{thm5.1} that $e_{\alpha}\in Center(\cal N)$ can be replaced by $\alpha+\beta \notin W$ for any $\beta\in W.$
\end{rem}

From the proof of Theorem \ref{thm5.1}, we obtain the following corollary.

\begin{cor}
Let $m,n$ and $p$ denote, respectively, the number of pairs satisfying \textbf{Case (i)} and \textbf{Case (ii)} and the number of triples satisfying \textbf{Case (iii)}. Then the dimension of the second cohomology group of $\cal R_{\cal T}$ is bounded below by $m+n+p$.
\end{cor}

Next, we present examples of algebras that satisfy exactly one of the conditions in Theorem~\ref{thm5.1} and an example of a Lie algebra that satisfies neither condition, yet whose second cohomology group is non-vanishing. This latter example shows that the hypotheses of the theorem, while sufficient for the non-vanishing of the second cohomology, are not necessary.

\begin{exa}\label{exa5.4} Consider the $9$-dimensional nilpotent Lie algebra $\cal N_9$ of maximal rank two whose only non-zero structure constants are the following (each equal to 1):
\begin{equation}\label{eq29}\left\{\begin{array}{llllll}
c(\alpha_1,\alpha_2),  \quad c(\alpha_1,\alpha_1+\alpha_2), \quad c(\alpha_1,\alpha_1+2\alpha_2), \quad  c(\alpha_1+\alpha_2,2\alpha_1+\alpha_2),\quad  c(\alpha_1+\alpha_2,\alpha_1+2\alpha_2),\\[3mm]
c(\alpha_1,2\alpha_1+\alpha_2), \quad c(\alpha_1,2\alpha_1+2\alpha_2),\quad c(\alpha_2,\alpha_1+\alpha_2), \quad c(\alpha_2,2\alpha_1+\alpha_2),\quad c(2\alpha_1+2\alpha_2,\alpha_2).\end{array}\right.\end{equation}
From \eqref{eq29} one can see that $\cal N_9$ satisfies only condition (i) of Theorem~\ref{thm5.1}. Consequently, for $11$-dimensional solvable Lie algebra $\mathcal{R}_{\mathcal{T}}$ of $\mathcal{N_9}$ we have $H^2(\mathcal{R}_{\mathcal{T}},\mathcal{R}_{\mathcal{T}}) \neq 0.$
\end{exa}

\begin{exa}
Consider the $10$-dimensional nilpotent Lie algebra $\cal N_{10}$ of maximal rank three whose only non-zero structure constants are the following (each equal to 1):
\begin{equation}\label{eq30}
\left\{\begin{array}{llllllllllllllll}
c(\alpha_1,\alpha_2), &c(\alpha_1,\alpha_1+\alpha_2), & c(\alpha_1,2\alpha_1+\alpha_2),& c(\alpha_1,\alpha_1+2\alpha_2),\\[3mm]
c(\alpha_1,3\alpha_1+\alpha_2),& c(\alpha_1,2\alpha_1+2\alpha_2), &  c(\alpha_1,3\alpha_1+2\alpha_2),& c(\alpha_2,\alpha_1+\alpha_2),\\[3mm]
c(\alpha_2,2\alpha_1+\alpha_2),& c(\alpha_2,4\alpha_1+\alpha_2),& c(\alpha_1+\alpha_2,2\alpha_1+\alpha_2),&c(\alpha_1+\alpha_2,3\alpha_1+\alpha_2).\end{array}\right.\end{equation}
 From \eqref{eq30}, we conclude that $\mathcal{N}_{10}$ does not satisfy condition (i) and (iii), but does satisfy condition (ii) of Theorem \ref{thm5.1}. It follows that the $12$-dimensional solvable Lie algebra $\mathcal{R}_{\mathcal{T}}$ constructed from $\mathcal{N}_{10}$ has non-trivial second cohomology.
\end{exa}

\begin{exa}\label{exa5.5} Consider the $7$-dimensional nilpotent Lie algebra $\cal N_7$ of maximal rank three whose only non-zero structure constants are the following (each equal to 1):
\begin{equation}\label{eq31} \begin{array}{cccccc}
c(\alpha_1,\alpha_2), & c(\alpha_1,\alpha_3), & 
c(\alpha_2,\alpha_3), & c(\alpha_1+\alpha_2, \alpha_3), & c(\alpha_1+\alpha_3,\alpha_2).
\end{array}\end{equation}
From \eqref{eq31}, we conclude that $\mathcal{N}_7$ does not satisfy condition (i) and (ii), but does satisfy condition (iii) of Theorem \ref{thm5.1}. It follows that the $9$-dimensional solvable Lie algebra $\mathcal{R}_{\mathcal{T}}$ constructed from $\mathcal{N}_{7}$ has non-trivial second cohomology.
\end{exa}

\begin{exa}\label{exa5.6} Consider the $12$-dimensional nilpotent Lie algebra $\cal N_{12}$ of maximal rank two whose only non-zero structure constants are the following (those in the first three rows equal to $1$, while in the last row equal to $\frac{1}{2}$)
$$\left\{\begin{array}{llllllllllllllll}
c(\alpha_1,\alpha_2), &c(\alpha_1,4\alpha_1+\alpha_2), & c(\alpha_1,4\alpha_1+2\alpha_2),\quad c(\alpha_1+\alpha_2,2\alpha_1+\alpha_2),\quad c(\alpha_2,3\alpha_1+\alpha_2),\\[3mm]
c(\alpha_1,3\alpha_1+\alpha_2),& c(\alpha_1,4\alpha_1+3\alpha_2), &  c(\alpha_2,3\alpha_1+2\alpha_2),\quad c(2\alpha_1+\alpha_2,3\alpha_1+\alpha_2), \quad  c(2\alpha_1+\alpha_2,\alpha_1),\\[3mm]
c(\alpha_2,4\alpha_1+\alpha_2),& c(\alpha_2,4\alpha_1+2\alpha_2),& c(\alpha_2,5\alpha_1+2\alpha_2),\quad c(2\alpha_1+\alpha_2,5\alpha_1+\alpha_2),\\[3mm]
c(\alpha_1+\alpha_2,\alpha_1),& c(\alpha_1,5\alpha_1+\alpha_2),&  c(3\alpha_1+\alpha_1,\alpha_1+\alpha_2),\quad  c(5\alpha_1+\alpha_2,\alpha_1+\alpha_2).
 \end{array}\right.$$

From the structure of the algebra $\cal N_{12}$ one can conclude that it does not satisfy the conditions (i), (ii) and (iii).
\end{exa}

Based on examples of solvable Lie algebras with non-vanishing second cohomology groups, we propose the following conjecture:

\begin{con}
Let $p$ denotes the number of weights $\Lambda$ for which the difference between the maximal number of pairwise distinct triples and the maximal number of distinct pairs associated with $\Lambda$ greater than $1$. Then
\[
H^{2}(\mathcal{R}_{\mathcal{T}},\mathcal{R}_{\mathcal{T}}) \;\geq\; p.
\]
\end{con}

{\bf Data availability} {{\small{\ Data sharing not applicable to this article as no datasets were generated or analysed during
the current study.}}
\\[3mm]

{\bf Conflict of interest} {{\small{\ On behalf of all authors, the corresponding author states that there is no conflict of interest.}}

\end{document}